\documentclass{article}
\usepackage{graphicx} 
\usepackage[T1]{fontenc}

\usepackage[a4paper]{geometry}
\usepackage{a4wide}
\usepackage{indentfirst}

\usepackage{amsthm}
\usepackage{amsmath}
\usepackage{amssymb}
\allowdisplaybreaks
\usepackage{xcolor}
\newtheorem{definition}{Definition}
\newtheorem{assumption}{Assumption}
\newtheorem{lemma}{Lemma}
\newtheorem{remark}{Remark}
\newtheorem{theorem}{Theorem}
\newtheorem{corollary}{Corollary}
\newtheorem{proposition}{Proposition}

\AtBeginEnvironment{algorithm}{\SetInd{.5em}{1em}}
\usepackage[boxruled]{algorithm2e}
\usepackage{algorithmic}

\usepackage{biblatex}
\bibliography{references}
\usepackage{mathrsfs}
\usepackage{hyperref}
\usepackage{csquotes}

\title{Variational Contraction Conditions for Iterative Algorithms in Multi-Population Discrete-Time Regularized Mean-Field Games\thanks{Research of the authors was supported in part by the AFOSR Grant FA9550-24-1-0152.}}
\author{U\u{g}ur Ayd{\i}n \and Tamer Ba\c{s}ar
\thanks{Authors are with the Department of Electrical and Computer Engineering, and the Coordinated Science Laboratory, University of Illinois Urbana-Champaign, Urbana, IL, 61801
        {\tt\small \{uaydin2,basar1\}@illinois.edu}}%
}

\begin{document}

\maketitle
\begin{abstract}
    In this work, we study the contraction conditions of iterative algorithms for stationary and finite-horizon discrete-time regularized mean-field games (MFGs) with multiple populations, where each population only interacts with the state distributions of the other populations. Due to the high dimensionality caused by the interaction of different populations, contraction rates for these algorithms cannot, in general, be expressed in terms of radicals. By studying the dynamics of these iterative algorithms and assuming that the system components of each population’s MFG are Lipschitz continuous, we present explicit (eventual) contraction conditions for each algorithm in any normed space, relying only on these Lipschitz parameters. As a consequence of these contraction conditions, we provide convergence rates of finite-horizon mean-field equilibria to infinite-horizon stationary (and nonstationary) mean-field equilibria (MFEs), under restrictions on a variational characterization of the dynamics of these iterative algorithms. In the single-population case, the restrictions we impose on this variational characterization to obtain these convergence results are less restrictive than previous results in the literature.
\end{abstract}
\section{Introduction}

In this work, we study the convergence of some basic iterative algorithms in discrete-time regularized mean-field games with multiple populations. We assume that interactions between populations are weak, in the sense that populations interact only through their state distributions, which is a standard setting in the literature on multi-population MFGs \cite{liu2025incorporating, wikecek2024multiple, uz2023reinforcement}. Our main goal is to identify the dynamical systems associated with these iterative algorithms and to study their contraction properties, which imply global asymptotic stability of the mean-field equilibrium (MFE). Beyond deterministic analyses, these global asymptotic stability results also enable the use of basic learning algorithms when components of the MFG are unknown, such as stochastic approximation \cite{borkar2000ode}.

In model-based learning settings \cite{huang2024statistical, pasztor2023efficient}, most methods are restricted to finite-horizon problems. In the single-population case, \cite[Theorem 4]{ayd} shows that learning a finite-horizon MFG with a sufficiently large horizon length can also be sufficient for the infinite-horizon setting, under additional assumptions on the system \cite[Corollaries 5 and 6]{ayd}. Another application of \cite[Theorem 4]{ayd} is a uniqueness condition for infinite-horizon, non-stationary MFEs obtained via finite-horizon MFEs, which is not directly available if one studies the infinite-horizon system itself (see \cite[Proposition 3]{ayd}). For these reasons, it is desirable to obtain contraction conditions in the multi-population setting, particularly for the finite-horizon case, since each of these results relies on establishing a horizon-length-independent contraction condition for finite-horizon MFGs.

The main difficulty in establishing suitable contraction conditions for multi-population MFGs is the presence of cross-interactions between populations. This cross-interaction, especially in the finite-horizon case, results in a high-dimensional dynamical system, which complicates the analysis. In our case, to study the contraction of this high-dimensional system under Lipschitz continuity assumptions, we will identify a matrix whose spectral radius (which will be an eigenvalue in this case) reveals the contraction property of the corresponding MFG. Since eigenvalues of matrices are roots of their characteristic equations, closed-form expressions for them are often not available in radicals, by a known theorem due to Abel and Ruffini. Furthermore, the matrices that appear in the finite-horizon case will be dense with no structure; thus, the runtime of brute-force calculation of such eigenvalues also scales with the dimension of the matrix, leading to calculating the spectral radius of an $N(T-1)\times N(T-1)$ dimensional matrix when there are $N$ populations and the horizon length is $T$. Moreover, due to the cross-interaction of these populations, the proofs of the results obtained in the single-population case do not directly scale to the multi-population setting. Thus, easily verifiable contraction conditions are desirable in the multi-population setting.

\subsection{Literature Review}

This work mainly extends the results of \cite{ayd} to the multi-population setting. The main distinction between our work and \cite{ayd} is that the proofs presented in \cite{ayd} do not scale to the multi-population setting. The contraction condition obtained in \cite[Theorem 2]{ayd} relies on proofs that are purely algebraic (except \cite[Lemma 9]{ayd}, which also admits a purely algebraic proof in the setting of that work). In this work, we will transfer the eigenvector problem associated with the matrix that characterizes contraction to the complex plane by using generating functions, and we will argue using complex-analytic methods that are common in the literature on analytic combinatorics \cite[Section VI]{flajolet2009analytic}. For continuous time MFGs, approximation of infinite-horizon MFE via finite-horizon is studied in \cite{carmona2024probabilistic} for the probabilistic solutions.

In discrete-time multi-population MFGs, there is very little work available. For example, \cite{uz2023reinforcement} studies basic RL algorithms in the setting of LQR average-cost MFGs in discrete time. Although contraction is used to analyze the learning algorithms, \cite[Theorem 3]{uz2023reinforcement} does not provide an easily verifiable condition under which the proposed contraction condition holds. In \cite{wikecek2024multiple}, the existence of an MFE is studied for multi-population discrete-time MFGs in the general case. The work \cite{xu2025q} applies model-free learning to stationary multi-population discrete-time MFGs that arise in networks. The work \cite{aggarwal2025semantic} also models a network problem as a multi-population discrete-time average-cost MFG and uses supervised learning to learn the MFE, under a blanket contraction assumption. In \cite{liu2025incorporating}, a multi-population MFG model is proposed to study epidemics in continuous time, where contraction arguments are used for learning the MFE. We note that all these works also consider weak interactions between populations.

Multi-time-scale analysis of iterative algorithms has also been studied in the discrete-time MFG setting \cite{angiuli2026analysis}. However, in that work, the global stability conditions required for convergence are justified solely in terms of the global Lipschitz properties of the MFG. In contrast, in Section~\ref{sect:7}, we show that, even in the multi-population setting, multi-time-scale analysis alone does not improve global contraction properties when one relies only on the Lipschitz properties of the dynamical system.

\subsection{Contributions}
In this paper, we present quantitative contraction conditions for multi-population stationary and finite-horizon MFGs in the presence of a regularizer.
\begin{enumerate}
    \item In Section~\ref{sect:stat}, we establish a contraction condition for a quasi-static iterative algorithm for regularized multi-population stationary MFGs; see Theorem~\ref{thrm:2}. At each iteration, the algorithm solves an MDP to compute the optimal $Q$-function of each population under the current population distribution $\tau$, and then updates the population distribution accordingly. We derive a condition that is equivalent to contractivity of this iteration and, consequently, guarantees convergence to a stationary MFE. In particular, this result extends \cite[Theorem~1]{anahtarci2023q} to the multi-population setting.
    \item In Section \ref{sect:var}, we present \emph{variational contraction conditions} for multi-population finite-horizon MFGs under regularization (Theorem \ref{thrm:4}), which is the main result of this work. The variational contraction condition that we provide reduces the problem of finding the spectral radius of the matrix to a one-dimensional optimization problem. In particular cases, for instance in the single-population case, the contraction result that we establish is equivalent to \cite[Theorem 2]{ayd}. Moreover, Theorem \ref{thrm:4} extends \cite[Theorem 2]{ayd} to the multi-population setting.
    \item In Section \ref{sect:6}, we present convergence rates and finite-time error bounds between finite-horizon MFEs and infinite-horizon MFEs (Theorem \ref{thrm:conv}).
    The variational contraction condition that we establish allows us to obtain convergence-rate results even in cases not covered by \cite[Theorem 4]{ayd} in the single-population setting. In particular, our variational contraction condition enables us to establish a convergence rate between finite-horizon and infinite-horizon MFEs whenever the stationary MFE is also contractive, which is not always possible under \cite[Theorem 4]{ayd}.
    \item In Section \ref{sect:7}, we show that contraction conditions established in Sections \ref{sect:stat} and \ref{sect:var} cannot be improved by iterating some populations on a faster time scale. Similarly, we show that the contraction condition established in Section \ref{sect:var} continues to hold when the $Q$-functions and state-measures are updated simultaneously, i.e., one does not need to solve the MDP as an intermediate step. These reductions rely on showing that the spectral radius of a nonnegative matrix remains invariant under block elimination. Nevertheless, the formulation in Section \ref{sect:stat} is still useful because it simplifies the derivation of a contraction condition.
\end{enumerate}

\section{Preliminaries}
In the mean-field setting, a \emph{population} is a continuum of homogeneous players. In our multi-population mean-field setting, we assume that there is a large number of players, each of whom belongs to a specific population. Each population has its own reward function and transition kernel.

Formally, a discrete-time \emph{multi-population mean-field game (MFG)} is described as follows:
\begin{enumerate}
    \item The game is played in discrete time, \(t \in \{1,2,\cdots\}\).
    \item The game is played by a continuum of players divided into \(N\) populations. For the sake of simplicity, we assume that each population shares the same finite state space \(X\) and finite action space \(A\).
    \item A family of state measures \(\tau_t = (\tau_{t,1},\tau_{t,2},\cdots,\tau_{t,N}) \in \prod_{j=1}^N \mathcal P(X)\) denotes the global state {distribution} of the game at time \(t\). The \(i\)-th component \(\tau_{t,i}\) describes the state distribution of population \(i\).

    We assume that at each stage \(t\), each player knows their state, their population's state measure, and the state measures of the other populations. Thus, the interaction between populations is \emph{weak}.
    \item The one-stage cost of an individual from population \(i\) is given by a function
    \(c_i: X \times A \times \prod_{j=1}^N \mathcal P(X) \to [0,\infty)\).
    The value \(c_i(x,a,\tau)\) is the cost incurred at any stage when the player is in state \(x\), takes action \(a\), and the joint state distribution of all populations is \(\tau\).
    \item State transitions for an individual in population \(i\) are defined by the stochastic kernel
    \(p_i: X \times A \times \prod_{j=1}^N \mathcal P(X) \to \mathcal P(X)\).
\end{enumerate}

With the conventions above, we represent a multi-population MFG by the tuple
\[
(X,A,(c_i)_{i \in [N]}, (p_i)_{i \in [N]},N),
\]
so that a multi-population MFG can be viewed as a collection of \(N\) distinct discrete-time MFGs. The main distinction between a multi-population MFG and a single-population MFG is the cross-interaction between different populations.

For a given \(t \in \mathbb N\) and \(i \in [N]\), a \emph{policy} is a stochastic kernel \(\pi_{t,i}:X \to \mathcal P(A)\) that specifies the distribution over actions for population \(i\) at time \(t\) given a state. We denote by \(\Pi\) the space of policies \((\pi_{t,i})_{t \in \mathbb N,\, i \in [N]}\). Policies being independent of the state measures of other populations is an artifact of the weak interaction assumption.

In general, a discrete-time MFG is an extension of discrete-time Markov decision processes (MDPs) \cite{SaBaRaSIAM}. Indeed, a MFG can be reduced to an MDP if one operates under a fixed family of state measures \((\tau_t)_{t=0}^{\infty}\). Let \((\beta_i)_{i \in [N]} \subset (0,1)\). Thus, in the infinite-horizon case, to define an objective function for each population, we need additional structure on the MFG.
For each \(i \in [N]\), \(\beta_i<1\) is the discount factor used by population \(i\). Under these discount factors, and for a given family of multi-population state measures \(\pmb{\tau} = (\tau_t)_{t=0}^{\infty}\), the discounted objective function {of the representative agent} of the population \(i \in [N]\) is defined as
\[
J_i(\pmb \pi, \pmb {\tau}) = \mathbb E_{\pmb \pi}\left[ \sum_{t=0}^{\infty}\beta_i^t\,c_i(x_t,a_t,\tau_{t,1},\tau_{t,2},\cdots,\tau_{t,N}) \right],
\]
where \(\pmb \pi = (\pi_{t,i})_{t \in \mathbb N,\, i\in[N]}\). For $J_i$, the evolutions of the states and actions are given by
\[
x(0) \sim \tau_{0,i}, \quad x(t) \sim p_i(\cdot | x(t-1), a(t-1), \tau_{t-1,i}), \quad t \geq 1, \quad a(t) \sim \pi_{t,i}(\cdot | x(t)), \quad t \geq 0.
\]

We say that a family of policies \(\pmb{\pi} := (\pi_{t,i})_{t \in \mathbb N,\, i \in [N]}\) is \emph{optimal} for a family of state measures \(\pmb{\tau} := (\tau_{t,i})_{t \in \mathbb N,\, i \in [N]}\) if, for all \(i \in [N]\),
\begin{equation}\label{eq:fgg}
\tau_{t+1,i}(\cdot) = \int_X \int_A p_i(\cdot \mid x,a,\tau_{t,1},\tau_{t,2},\cdots,\tau_{t,N})\,\pi_{t,i}(da\mid x)\,\tau_{t,i}(dx).
\end{equation}
Conversely, we say that a family of state measures \(\pmb{\tau}\) is \emph{optimal} for a family of policies \(\pmb{\pi}\) if, for all \(i \in [N]\),
\[
J_i(\pmb{\pi}, \pmb{\tau}) = \inf_{\tilde{\pmb{\pi}} \in \Pi} J_i(\tilde{\pmb{\pi}}, \pmb{\tau}).
\]

A solution to the MFG consists of a family of policies \((\pi_{t,i})_{t \in \mathbb N,\, i \in [N]}\) and state measures \((\tau_{t,i})_{t \in \mathbb N,\, i \in [N]}\) such that \((\pi_{t,i})_{t \in \mathbb N,\, i \in [N]}\) is optimal for \((\tau_{t,i})_{t \in \mathbb N,\, i \in [N]}\) and \((\tau_{t,i})_{t \in \mathbb N,\, i \in [N]}\) is optimal for \((\pi_{t,i})_{t \in \mathbb N,\, i \in [N]}\). In this case, we call \(\bigl((\pi_{t,i})_{t \in \mathbb N,\, i \in [N]},(\tau_{t,i})_{t \in \mathbb N,\, i \in [N]}\bigr)\) a \emph{mean-field equilibrium} (MFE).

In this work, we are interested in computing MFEs through iterative algorithms. For this purpose, we will mainly work with \emph{deterministic policies}. A policy \(\pmb{\pi} = (\pi_{t,i})_{t\in \mathbb N,\, i \in [N]} \in \Pi\) is said to be deterministic if, for all \(t \in \mathbb N\) and \(i \in [N]\), there exists a function \(f_{t,i}:X \to A\) such that \(\pi_{t,i}(\cdot\mid x) = \delta_{f_{t,i}(x)}(\cdot)\), where \(\delta\) denotes the Dirac measure. The space of all deterministic policies is then denoted by \(\Pi_d\).

To compute an MFE, throughout the iterations, we will need to ensure that there exists a unique Lipschitz policy. To achieve this, we first lift the action space of the MFG to \(\mathcal P(A)\). This amounts to defining an MFG \((X,\mathcal P(A),(C_i)_{i \in [N]}, (P_i)_{i \in [N]},N)\) that shares the ``same'' MFE as \((X,A,(c_i)_{i \in [N]}, (p_i)_{i \in [N]},N)\). First, we embed \(\mathcal P(A) \subset \mathbb R^{|A|}\). For any \(u \in \mathcal P(A)\), we define
\[
C_i(x,u,\tau)= C_i(x,u,\tau_1,\cdots,\tau_N)= \sum_{a \in A} c_i(x,a,\tau_1,\cdots,\tau_N)\,u(a),
\]
and
\[
P_i(\cdot\mid x,u,\tau)=P_i(\cdot\mid x,u,\tau_1,\cdots,\tau_N)= \sum_{a \in A} p_i(\cdot\mid x,a,\tau_1,\cdots,\tau_N)\,u(a).
\]
{Thus, $C_i$ and $P_i$ are affine in $\mathcal P(A)$.} Now, to ensure that we obtain unique Lipschitz policies during our iterations, we perturb each \(C_i\) as \(C_i+\rho_i\Omega_i\), where \(\rho_i>0\) and \(\Omega_i : \mathcal P(A) \to \mathbb R\). We further assume that each \(\Omega_i\) is \(1\)-strongly convex with respect to the \(\ell_1\)-norm, i.e.,
\[
\Omega_i(\theta u+(1-\theta)v) \leq \theta \Omega_i(u) +(1-\theta)\Omega_i(v) - \frac 12 \theta(1-\theta) \| u-v\|_1^2
\]
for all \(u,v \in \mathcal P(A)\) and \(\theta \in [0,1]\).
The resulting MFG \((X,\mathcal P(A),(C_i+\rho_i\Omega_i)_{i \in [N]}, (P_i)_{i \in [N]},N)\) is called a \emph{regularized MFG}, and the parameters \(\rho_i\) are called \emph{regularization parameters}. {Since $C_i$ and $P_i$ are affine, the $Q$-functions of regularized MFGs are strongly convex over $\mathcal P(A)$ due to the regularizer.} The MFE obtained from the regularized MFG deviates from that of the original lifted MFG; however, these perturbations ensure the existence of a unique Lipschitz flow of policies that achieves the objective function for any given family of state measures. Furthermore, MFE of the regularized MFGs obtained only under deterministic policies.

To complete our setup, we impose a Lipschitz continuity assumption on the system components \(c_i\) and \(p_i\) as follows:
\begin{assumption}\label{ass:1}
    We denote by \(\|\cdot\|_1\) the \(\ell_1\)-norm on \(\mathcal P(X) \subset \mathbb R^{|X|}\). For all \(i \in [N]\), \(x,\tilde x \in X\), \(a,\tilde a \in A\), and \(\tau=(\tau_i)_{i\in[N]},\tilde \tau = (\tilde \tau_i)_{i \in [N]} \in \prod_{i \in [N]}\mathcal P(X)\), we have
    \[
    \left| c_i(x,a,\tau)-c_i(\tilde x, \tilde a, \tilde \tau) \right| \le L_i \left( \mathbf 1_{\{ x \neq \tilde x\}} + 2\,\mathbf 1_{\{a \neq \tilde a\}}+ \sum_{j=1}^N\| \tau_j-\tilde \tau_j\|_1\right),
    \]
    and
    \[
    \left \| p_i(\cdot\mid x,a,\tau) - p_i(\cdot\mid \tilde x,\tilde a,\tilde \tau) \right\|_1 \le K_i\left( \mathbf 1_{\{x \neq \tilde x\}}+ 2\,\mathbf 1_{\{a \neq \tilde a\}} + \sum_{j=1}^N \| \tau_j - \tilde \tau_j\|_1  \right).
    \]
\end{assumption}
For the rest of the paper, we will suppose that Assumption \ref{ass:1} is effective. Lipschitz continuity assumptions such as the ones above are common in the analysis of iterative algorithms \cite{anahtarci2023learning, guo2019learning, cui2021approximately}.
For the purpose of our iteration algorithms, we define a candidate domain for our \(Q\)-functions at the multi-population level as follows:
Let $M_i$ be a uniform bound for the cost function $c_i$, {which is finite by Assumption \ref{ass:1}}. Under the MFG $(X,A,(c_i)_{i \in [N]}, (p_i)_{i \in [N]},N)$, the candidate space of $Q$-functions for the population $i$ over the action space $A$ is
\[
\tilde{\mathcal C}_i := \Big\{
    Q^{\mathrm{disc}}:X\times A \to \mathbb R_{+}:\;
    |Q^{\mathrm{disc}}(x,a)-Q^{\mathrm{disc}}(\tilde{x},\tilde{a})|
    \le \frac{L_i}{1 - \frac{\beta_i K_i}{2}}(1_{x \not=\tilde x}+2\, 1_{a \not = \tilde a}),\;
    \|Q\|_\infty \le \frac{M_i}{1-\beta_i}
\Big\}.
\]
Then, {to lift these $Q$-functions to the regularized game, we define the set}:
\[
\mathcal C_i :=
\Big\{
    Q : X \times \mathcal P(A) \to \mathbb R:\;
    Q(x,u)= \sum_{a \in A} Q^{\mathrm{disc}}(x,a)u(a) + \rho_i\Omega_i(u),
    Q^{\mathrm{disc}} \in \tilde{\mathcal C}_i
\Big\}.
\]
For a given \(Q^{\mathrm{disc}} \in \tilde{\mathcal C}_i\), the corresponding \(Q^{\mathrm{cont}} \in \mathcal C_i\) given by \(Q^{\mathrm{cont}}(x,u) =\sum_{a \in A} Q^{\mathrm{disc}}(x,a)u(a) + \rho_i\Omega_i(u)\) defines a one-to-one correspondence. Furthermore, \(Q^{\mathrm{cont}}\) is the \(Q\) function one would obtain under \((X,\mathcal P(A),(C_i+\rho_i\Omega_i)_{i \in [N]}, (P_i)_{i \in [N]},N)\). Note that each \(Q^{\mathrm{cont}} \in \mathcal C_i\) have a unique minimizer {in $\mathcal P(A)$}; thus, the policy \(\pi_{Q^{\mathrm{cont}}}\) that one obtains from the minimizers of $Q^{\mathrm{cont}}(x,\cdot)$ is a deterministic policy for the MFG \((X,\mathcal P(A),(C_i+\rho_i\Omega_i)_{i \in [N]}, (P_i)_{i \in [N]},N)\). {In what follows, for a given $Q\in \mathcal C_i$, we will also let $Q_{\min}(x):= \operatorname{min}_{u \in \mathcal P(A)}Q(x,u).$}

The algorithms that we consider in this work will closely mimic the classical value-iteration algorithm for MDPs. For a given family of state measures \(\pmb{\tau}\), we update our \(Q\)-functions as in value iteration, and, to obtain the consistency condition \eqref{eq:fgg}, we also need to control the updates of our state measures. For this purpose, we define the following operators:

\begin{definition}
    Let \(Q \in \mathcal C_i\) and \(\tau = (\tau_i)_{i=1}^N \in \prod_{i=1}^{N} \mathcal P(X)\). For updating our \(Q\)-functions, we define the operator
    \[
    H_{1,i}(Q,\tau)(x,u) = C_i(x,u,\tau) + \rho_i\Omega_i(u)+ \beta_i \sum_{y \in X}Q_{\min}(y)\,P_i(y\mid x,u,\tau),
    \]
    and for updating our state measures, we define
    \[
    H_{2,i}(Q,\tau)(\cdot)= \sum_{x \in X}\int_{\mathcal P(A)} P_i(\cdot\mid x,u,\tau)\,\pi_Q(du\mid x)\,\tau_i(x),
    \]
    where {\(\pi_Q(\cdot \mid x) = \delta_{\operatorname{argmin}_{u \in \mathcal P(A)}Q(x,u)}(\cdot)\) is the optimal policy induced by the \(Q\)-function \(Q\). Since $Q \in \mathcal C_i$, $\pi_Q$ is well-defined.}
\end{definition}

The operator \(H_{1,i}\) arises as a consequence of the dynamic programming structure associated with the objective functions \(J_i\), a well-known result in the theory of MDPs. The operator \(H_{2,i}\) is specific to the MFG setting and arises from the consistency requirement linking policies and state measures through cross-population interactions. From the definition of \(H_{1,i}\), note that since \(C_i\) and \(P_i\) are affine over \(\mathcal P(A)\) and \(\Omega_i\) is strongly convex, for any fixed \(Q\)-function there exists a unique minimizer of \(H_{1,i}(Q,\tau)(x,\cdot)\) for each \((x,\tau)\). In particular, the operator \(H_{2,i}\) is well defined.

In this work, we will deal with three different types of MFE. In the infinite-horizon case, an MFE \((\pmb{\pi}, \pmb{\tau})\) is said to be \emph{stationary} if it is time independent, i.e., \(\tau_{t,i}=\tau_{t+1,i}\) and \(\pi_{t,i}=\pi_{t+1,i}\) for all \(i \in [N]\) and all \(t \in \mathbb N\). Due to this stationarity condition, for a stationary MFE we do not impose any restriction on the initial state measure. If the initial state measures are not fixed, one often seeks a stationary MFE.

Given a family of initial state measures \(\tau_0=(\tau_{0,i})_{i \in [N]} \in \prod_{i \in [N]}\mathcal P(X)\), we say that an MFE \((\pmb{\pi}, (\tilde{\tau}_{t,i})_{t \in \mathbb N,\, i \in [N]})\) is \emph{non-stationary} if \(\tilde{\tau}_{0,i} = \tau_{0,i}\) for all \(i \in [N]\). In the non-stationary case, one fixes the initial state measures \emph{a priori}, which usually come from the population's initial distribution. For an initial state measure \(\tau_0 \in \prod_{i \in [N]}\mathcal P(X)\), we denote by \(\mathrm{MFE}_{\tau_0}\) the set of MFEs for the infinite-horizon MFG starting from \(\tau_0\).

In the finite-horizon case, one often proceeds as in the infinite-horizon case, with the caveat that the objective function {of the representative agent} takes the form
\[
J_i(\pmb \pi, \pmb {\tau}) = \mathbb E_{\pmb \pi}\left[ \sum_{t=0}^{T-1}\beta_i^t\,c_i(x_t,a_t,\tau_{t,1},\tau_{t,2},\cdots,\tau_{t,N}) \right].
\]
The truncation of the sum at time \(T\) imposes a natural boundary condition on the iterations of the \(Q\)-functions. In this case, we denote by \(\mathrm{MFE}_{T,\tau_0}\) the set of finite-horizon MFEs starting from \(\tau_0\).

Although we have defined the operators \(H_{1,i}\) and \(H_{2,i}\) for the lifted system, Assumption \ref{ass:1} above is stated for the original system. It turns out, an analogous Lipschitz continuity condition also holds for the lifted system \cite{anahtarci2023q}. 
\begin{lemma}\label{lem:70}
    For all \( x, \tilde{x} \in X \), \( u, \tilde{u} \in \mathcal P(A) \), and \( \tau, \tilde{\tau} \in \prod_{i \in [N]}\mathcal P(X) \), $C_i$ and $P_i$ satisfy the following Lipschitz bounds:
\[
|C_i(x, u, \tau) - C_i(\tilde{x}, \tilde{u}, \tilde{\tau})| \leq L_i
\left( \mathbf{1}_{\{x \neq \tilde{x}\}} + \|u - \tilde{u}\|_1 + \sum_{j \in [N]}\|\tau_j - \tilde{\tau}_j\|_{1} \right),
\]
and
\[
\|P_i(\cdot|x, u, \tau) - P_i(\cdot|\tilde{x}, \tilde{u}, \tilde{\tau})\|_{1} \leq K_i
\left( \mathbf{1}_{\{x \neq \tilde{x}\}} + \|u - \tilde{u}\|_1 + \sum_{j \in [N]}\|\tau_j - \tilde{\tau}_j\|_{1} \right).
\]
\end{lemma}
\begin{proof}
    The same proof in \cite{anahtarci2023q} holds verbatim, so we omit the details.
\end{proof}

An important consequence of Lemma \ref{lem:70} is the following Lipschitz property of \(H_{1,i}\) and \(H_{2,i}\), which will be fundamental for deriving a variational expression for our iterations in the stationary and finite-horizon settings. We will provide a proof for such an expression by closely following the proof of the single population case presented in \cite{anahtarci2023q}. Since our main results will heavily rely on these variational inequalities, we will provide a detailed proof for completeness in the appendix.

\begin{lemma}\label{lem:1}
    For any $Q,\tilde Q \in \mathcal C_i$ and $\tau, \tilde \tau \in \prod_{i=1}^N \mathcal P(X)$ we have
    \begin{equation}\label{eq:q}
    \| H_{1,i}(Q,\tau) - H_{1,i}(\tilde Q, \tilde \tau) \|_{\infty} \le \beta_i \| Q-\tilde Q\|_{\infty}+ \frac{L_i}{1-\frac{\beta_i K_i}{2}}\left( \sum_{j=1}^N\| \tau_j-\tilde \tau_j\|_1 \right),
    \end{equation}
    and
    \begin{equation}\label{eq:m}
    \| H_{2,i}(Q,\tau) - H_{2,i}(\tilde Q,\tilde \tau) \|_1 \le \frac{K_i}{\rho_i}\|Q-\tilde Q\|_{\infty} + \overline K_i \|\tau_i-\tilde \tau_i\|_1 + K_i\sum_{j \not = i}\| \tau _j - \tilde \tau_j\|_1,
    \end{equation}
    where 
    \begin{equation}\label{eq:fig}
    \overline K_i = \frac 32K_i + \frac 12 \frac{\overline L_i K_i}{\rho_i(1-\beta_i)}, \text{ and } \overline L_i = \frac{L_i}{1-\frac{\beta_i K_i}{2}}.
    \end{equation}
    
\end{lemma}
\begin{proof}
    We postpone the proof to Appendix \ref{app:a}.
\end{proof}

We also note that for a given $\tau \in \prod_{i \in [N]} \mathcal P(X)$, $H_{1,i}(\cdot,\tau)$ maps $\mathcal C_i$ to itself.

\begin{lemma}\label{lem:tency}
    For each $i \in [N]$, the operator $H_{1,i}(\cdot,\tau)$ maps $\mathcal C_i$ to itself.
\end{lemma}
\begin{proof}
    The same arguments in the proofs \cite[Lemma 2]{ayd} and \cite[Lemma 1]{anahtarci2020value} are directly applicable here; so, we will omit the details.
\end{proof}

\section{Eventual Contractivity of a Quasi-Static Iterative Algorithm for Multi-population Discounted Stationary MFGs}\label{sect:stat}

Let
$
\bigl(X,\mathcal{P}(A),(C_i+\rho_i\Omega_i)_{i\in[N]},(P_i)_{i\in[N]},N\bigr)
$
be a regularized multi-population MFG, where the objective function of {the representative agent of} population \(i\) is discounted by the factor \(\beta_i\), for each \(i\in[N]\). In this section, we establish a contraction condition to an iterative algorithm to compute the MFE of this multi-population stationary MFG. Since we are looking for stationary equilibria, the multi-population case amounts to dealing with \(N\) interacting dynamical systems. Since stationary solutions are time-independent, contraction conditions for stationary MFE induce a simpler structure to analyze compared to the non-stationary setting, in both the finite- and infinite-horizon cases.

The iterative algorithm in the stationary MFE case involves updating both \(Q\)-functions (using the current state measures as input) and state measures (using the minimizer induced by the current \(Q\)-functions). For simplicity, we will study the quasi-static case, where one finds a fixed point of $H_{1,i}(\cdot,\tau)$ at each step for a given initial state-measure $\tau \in \prod_{i \in [N]}\mathcal P(X)$. In practice, when the model is unknown, these $Q$-functions can be calculated using stochastic approximation \cite{guo2019learning}. We refer to this as the quasi-static case since one updates \(Q\)-functions faster than \(\tau\) to obtain a fixed-point \cite{borkar1997stochastic}.

First, to provide a formal definitions of this algorithm, we introduce the following notation:

\begin{definition}
For a given
$
\tau \in \prod_{i=1}^N \mathcal P(X),
$
let \(Q^{\tau,i} \in \mathcal C_i\) denote the \(Q\)-function satisfying
\begin{equation}\label{eq:bellman}
H_{1,i}(Q^{\tau,i},\tau)=Q^{\tau,i}.
\end{equation}
Since \(\max_{i\in[N]}\beta_i<1\), this fixed point exists and is unique for every \(i\in[N]\), as a consequence of the contraction property of the Bellman operator. Moreover, \(Q^{\tau,i}\in\mathcal C_i\) by Lemma \ref{lem:tency}.
\end{definition}

When the population \(i\) is clear from the context, we will also refer to \(Q^{\tau,i}\) simply as the \(Q\)-function corresponding to \(\tau\), somewhat abusing the notation. An iterative algorithm that finds a stationary MFE in this case is then the following:

\begin{algorithm}[H]\label{alg:1}
\caption{Quasi-Static Value Iteration for Multi Population MFG}
    Initialize with $\pmb \tau_0 = (\tau_{i,0})_{i \in [N]} \in \prod_{i \in [N]} \mathcal P(X)$\;
    \While{$\pmb \tau_{n+1} \not = \pmb \tau_n$}{
        $\pmb \tau_{n+1} = (H_{2,i}(Q^{\pmb \tau_{n},i},\pmb \tau_n))_{i \in [N]}$\;
    }
    Let $\pi^*_i = \delta_{\mathrm{argmin}_{u \in \mathcal P(A)}Q^{\pmb \tau^{*},i}(\cdot,u)}(\cdot)$ for all $i \in [N]$\;
    Let $\pmb \pi^* = (\pi^*_i)_{i \in [N]}$.\;
\Return{stationary mean-field equilibrium $(\pmb  \pi^*, \pmb \tau^*)$}
\end{algorithm}

As described in the algorithm, the fixed points of the map $\tau \mapsto (H_{2,i}(Q^{\tau,i},\tau))_{i \in [N]}$ induce stationary MFE. To study the convergence of this algorithm, we will provide an (eventual) contraction condition for $\tau \mapsto (H_{2,i}(Q^{\tau,i},\tau))_{i \in [N]}$.

In the next lemma, we will provide a matrix inequality for the map $(H_{2,i})_{i \in [N]}$ that will be fundamental while obtaining an eventual contraction condition that ensures the convergence of Algorithm \ref{alg:1}.

\begin{lemma}\label{lem:2}
For all $\tau,\tilde \tau \in \prod_{i \in [N]}\mathcal P(X)$, we have
\[
    \begin{bmatrix}
        \|H_{2,1}(Q^{\tau,1},\tau)-H_{2,1}(Q^{\tilde \tau,1},\tilde \tau)\|_1\\
        \|H_{2,2}(Q^{\tau,2},\tau)-H_{2,2}(Q^{\tilde \tau,2},\tilde \tau)\|_1\\
        \vdots \\
        \|H_{2,N}(Q^{\tau,N},\tau) - H_{2,N}(Q^{\tilde \tau,N},\tilde \tau)\|_1
    \end{bmatrix} \le \mathcal M\begin{bmatrix}
        \| \tau_1-\tilde \tau_1\|_1\\
        \| \tau_2 - \tilde \tau_2\|_1\\
        \vdots \\
        \| \tau_N-\tilde \tau_N\|_1
    \end{bmatrix},
\]
where $\mathcal M = \begin{bmatrix} M_{i,j} \end{bmatrix}_{i,j}$ is such that
$
M_{i,j} =
\begin{cases}
\overline K_i + \dfrac{\overline L_i K_i}{\rho_i(1-\beta_i)},
& i=j, \\[8pt]

K_i+\dfrac{\overline L_i K_i}{\rho_i(1-\beta_i)},
&  j \neq i. \\[8pt]
\end{cases}
$
\end{lemma}
\begin{proof}
    For any $j\in [N]$, using Lemma \ref{lem:1} we obtain
    \[
    \|H_{2,j}(Q^{\tau,j},\tau)-H_{2,j}(Q^{\tilde \tau,j},\tilde \tau)\|_1 \le \overline K_j \| \tau_j-\tilde \tau_j\|_1 + \frac{K_j}{\rho_j} \|Q^{\tau,j}-Q^{\tilde \tau, j}\|_{\infty}+K_j\left( \sum_{i \not = j}\| \tau_i - \tilde \tau_i\|_1\right).
    \]
    Using the Bellman equation \eqref{eq:bellman} and Lemma \ref{lem:1}, it also follows that
    \[
    \| Q^{\tau,j}-Q^{\tilde \tau,j} \|_{\infty} \le \frac{1}{1-\beta_j} \overline L_j\left( \sum_{i \in [N]} \|\tau_i-\tilde \tau_i\|_1 \right).
    \]
    Combining these two inequalities, we get
    \begin{align*}
\|H_{2,j}(Q^{\tau,j},\tau)-H_{2,j}(Q^{\tilde \tau,j},\tilde \tau)\|_1 &\le \left(\overline K_j + \frac{\overline L_j K_j}{\rho_j(1-\beta_j)} \right)\| \tau_j - \tilde \tau_j\|_1 
\\&\qquad+ \left( K_j + \frac{\overline L_j K_j}{\rho_j(1-\beta_j)}\right)\left( \sum_{i \not = j} \| \tau_i - \tilde \tau_i \|_1\right),
    \end{align*}
    as desired
\end{proof}

By $\rho(\mathcal M)$, we denote the spectral radius of the matrix $\mathcal M$. Let \(\|\cdot\|\) be a matrix norm on \(\mathbb R^{N \times N}\). As a direct consequence of Gelfand's formula, we have
\[
\lim_{n \to \infty} \| \mathcal M^n\|^{1/n} = \rho(\mathcal M),
\]
and this bound is typically only asymptotically tight for most matrix norms. Thus, if \(\rho(\mathcal M)<1\), then the map \((H_{2,i})_{i \in [N]}\) is eventually contractive under any norm on \(\prod_{i \in [N]} \mathcal P(X)\). With this observation, and using Lemma \ref{lem:2}, we obtain the following meta contraction condition in the stationary case:

\begin{theorem}\label{thrm:0}
    If $\rho(\mathcal M)<1$, then $\tau \mapsto (H_{2,i}(Q^{\tau,i},\tau))_{i \in [N]}$ is eventually contractive over $\prod_{i \in [N]} \mathcal P(X)$ under any norm, and thus Algorithm \ref{alg:1} converges.
\end{theorem}
\begin{proof}
    This follows from the discussion above.
\end{proof}

In general, finding a simple closed-form expression for $\rho(\mathcal M)$ in terms of radicals is impossible (due to Abel-Ruffini theorem) as it amounts to finding a root of the characteristic equation of $\mathcal M$, which has degree $|N|$. However, since $\mathcal M$ is a nonnegative irreducible matrix, spectral radius $\rho(\mathcal M)$ of $M$ corresponds to an eigenvalue of $\mathcal M$ by the Perron-Frobenius theorem, which can be used to give a full characterization of $\rho(\mathcal M)<1$.

\begin{theorem}\label{thrm:2}
    We have $\rho(\mathcal M)<1$ if and only if 
    \[
    \sum_{i \in [N]} \left(K_i+\frac{\overline L_i K_i}{\rho_i(1-\beta_i)}\right)\frac{1}{(1-(\overline K_i-K_i))}<1
    \]
    and \(\max_{i \in [N]} (\overline K_i - K_i)<1\).
\end{theorem}
\begin{proof}
    Let $h$ be a positive right eigenvector such that $\mathcal Mh = \rho(\mathcal M)h$. Then, using the equality of $\mathcal Mh = \rho(\mathcal M)h$ on the $j$'th row, one can see that for all $j\in[N]$ we have
    \begin{equation}\label{ff}
    (\overline K_j-K_j)h_j + \left(K_j+\frac{\overline L_j K_j}{\rho_j(1-\beta_j)}\right)\underbrace{\sum_{i \in [N]}h_i}_{=k} = \rho(\mathcal M)h_j.
    \end{equation}
    Note that, for all $j \in [N]$ we have $\rho(\mathcal M)>\overline K_j-K_j$ due to monotonicity of the spectral radius for nonnegative matrices. Thus, for all $j\in [N]$ we have
    \begin{equation}\label{ff1}
    \frac{K_j+\frac{\overline L_j K_j}{\rho_j(1-\beta_j)}}{\rho(\mathcal M)-(\overline K_j-K_j)}k = h_j.
    \end{equation}
    Equation \eqref{ff1} together with the definition of $k$ on \eqref{ff} imply that
    \begin{equation}\label{ff2}
    k = \sum_{i \in [N]} \frac{\left(K_i + \frac{\overline L_iK_i}{\rho_i(1-\beta_i)}\right)}{\rho(\mathcal M)-(\overline K_i-K_i)}k.
    \end{equation}
    The equation \eqref{ff2} is feasible if and only if
    \begin{equation}\label{fg}
    \sum_{i \in [N]} \frac{\left(K_i + \frac{\overline L_iK_i}{\rho_i(1-\beta_i)}\right)}{\rho(\mathcal M)-(\overline K_i-K_i)} = 1.
    \end{equation}
    Note that the function
    \begin{equation}\label{ff3}
    x \mapsto \sum_{i \in [N]} \frac{\left(K_i + \frac{\overline L_iK_i}{\rho_i(1-\beta_i)}\right)}{x-(\overline K_i-K_i)}
    \end{equation}
    is monotone decreasing on $(\max_j(\overline K_j-K_j),\infty)$; thus, there exists a unique $x > \max_j (\overline K_j-K_j)$ that satisfies \eqref{fg}, which is $\rho(\mathcal M)$. Using the monotonicity of \eqref{ff3} together with \eqref{fg}, one can directly deduce that the function defined by \eqref{ff3} is less than $1$ at $x=1$ is equivalent to $\rho(\mathcal M)<1$.
\end{proof}

\section{Variational Contraction Conditions for Finite-Horizon MFG}\label{sect:var}

Let $\mathrm{MFG}_{\mathrm T}=(X,\mathcal{P}(A),(C_i+\rho_i\Omega_i)_{i \in [N]}, (P_i)_{i \in [N]},N)$ be a finite-horizon regularized discrete-time MFG with horizon-length $T$. In this section, we will study the convergence behavior of a simple iterative algorithm to find a MFE for $\mathrm{MFG}_{\mathrm T}$ under a given initial state-measure $\tau_0$.

\subsection{Statements of Main Results}\label{sect:5.1}

In this subsection, we construct a matrix whose spectral radius yields a contraction condition in the finite-horizon setting, in a manner similar to Section \ref{sect:stat}. Due to the multi-population setting and the non-stationarity of the solutions, the relevant contraction condition is determined by the spectral radius of an \(N(T-1)\times N(T-1)\) dense matrix with no exploitable structure. We will derive a majorant for this contraction matrix that is independent of the horizon length and show that the spectral radius of these contraction matrices converges to that of the majorizing matrix, which will allow us to provide a horizon length independent contraction condition for MFG$_{\mathrm T}$.

For a given \(\pmb{\tau} = (\tau_t)_{t=0}^{T-1} \in \prod_{t=0}^{T-1}\left(\prod_{i=1}^N \mathcal{P}(X)\right)\), let \(\tau_t\) denote the section of the flow \(\pmb{\tau}\) at time step \(t\), regarded as an element of \(\prod_{i=1}^N \mathcal{P}(X)\). Similarly, for a given \(t \in \{0, \cdots, T-1\}\), let \(\tau_{t,j}\) denote the component of \(\tau_t\) corresponding to population \(j\), regarded as an element of \(\mathcal{P}(X)\).
\begin{definition}
Let \(\pmb \tau =(\tau_{t})_{t=1}^{T-1}\in \prod_{t=0}^{T-1}\left(\prod_{i=1}^N \mathcal P(X)\right)\). For each \(i \in [N]\), we define a sequence of \(Q\)-functions \((Q^{\tau,i}_t)_{t=0}^{T-1}\) by backward recursion, starting from \(t=T-1\):
\[
Q^{\pmb\tau,i}_{T-1}(x,u) = C_i(x,u,\tau_{T-1})+\rho_i \Omega_i(u),
\]
and, for \(t=0,1,\ldots,T-2\),
\[
Q^{\pmb \tau,i}_t(x,u) = C_i(x,u,\tau_{t}) + \rho_i \Omega_i(u) + \beta_i \sum_{y \in X} Q^{\pmb \tau,i}_{\min,t+1}(y)\,P_i(y\mid x,u,\tau_{t}).
\]
We denote by \(\pmb Q^{\tau,i} =(Q^{\tau,i}_t)_{t=0}^{T-1}\) the resulting \(Q\)-function flow. Note that, by Lemma \ref{lem:tency}, we have $Q^{\tau,i}_t \in \mathcal C_i$ for all $i \in [N]$ and $t = 0, \ldots, T-1$.
\end{definition}

Note that the flows $(\pmb Q^{\tau,i})_{i \in [N]}$ are obtained under the discount factor $\beta<1$. We will need the discount factor to analyze the long time behavior of finite-horizon MFE to identify a limiting system, which will be an infinite-horizon MFGs in our case. An iterative algorithm that finds a finite-horizon MFE starting from the initial state measure $\tau_0 \in \prod_{i \in [N]}\mathcal P(X)$ is the following:

\begin{algorithm}[H]\label{alg:2}
\caption{Iteration Algorithm for Finding Finite-Horizon MFE}
    Let $\pmb \tau_0 = (\tau_{i,0})_{i \in [N]} \in \prod_{i \in [N]} \mathcal P(X)$ be an initial state-measure\;
    Let $\pmb \tau^0 := (\tau_{0},\tau_{1},\cdots, \tau_{T-1}) \in \prod_{t=0}^{T-1}(\prod_{i \in [N]}\mathcal P(X))$\;
    \While{$\pmb \tau^{n+1} \not = \pmb \tau^n$}{
        $\pmb \tau^{n+1} = (\tau_{0,i},H_{2,i}(Q^{\pmb \tau^{n},i}_0, \tau^n_{0}),H_{2,i}(Q^{\pmb \tau^n,i}_{1},\tau^n_1)_{i \in [N]},\cdots,H_{2,i}(Q^{\pmb \tau^{n},i}_{T-1},\tau^n_{T-1}))_{i \in [N]}$\;
    }
    Define $\pi^*_{t,i} = \delta_{\mathrm{argmin}_{u \in \mathcal P(A)}Q_t^{\pmb \tau^{*},i}(\cdot,u)}(\cdot)$ for all $i \in [N]$ and $t=0,\cdots,T-1$\;
    Let $\pmb \pi^* = (\pi^*_i)_{i \in [N]}$\;
\Return{finite-horizon mean-field equilibrium $(\pmb  \pi^*, \pmb \tau^*)$}
\end{algorithm}

The contraction map that we will consider that ensures converge of Algorithm \ref{alg:2} in this section will take a family of state-measure $\tau$ as an input, calculate $\pmb Q^{\tau,i}$ for all $i\in [N]$, and then update $\tau$ by using the minimizers of $\pmb Q^{\tau,i}$ as required for an MFE.

The main result of this section is the following contraction theorem, which provides a variational condition guaranteeing the contraction uniformly in the horizon length.

\begin{theorem}\label{thrm:quartic}
     Let $\overline K_{\infty} =\max_j(\overline K_j-K_j)$. If
     \begin{equation}\label{eq:cont}
    \inf_{ \overline K_\infty < r < \beta_{\max}^{-1}} \sum_{i \in [N]}\frac{\frac{K_i+\frac{\overline L_i K_i}{\rho_i}}{r}+\frac{\overline L_i K_i \beta_i}{\rho_i(1-\beta_i r)}}{1-\frac{\overline K_i-K_i}{r}} <1,
    \end{equation}
    then for all $T \in \mathbb N$, the map $\pmb \tau \mapsto \left((\tau_{0,i})_{i=1}^N,\left(H_{2,i}(Q^{\tau,i}_t,\tau_t)\right)_{i=1,t=1}^{N,T-1}\right)$ is eventually contractive over $\prod_{t=1}^{T-1}\left(\prod_{i=1}^N \mathcal P(X)\right)\times \prod_{i=1}^N\{\tau_{0,i}\}$ for any $(\tau_{0,i})_{i=1}^N\in \prod_{i=1}^N \mathcal P(X)$. In particular, when \eqref{eq:cont} holds, Algorithm \ref{alg:2} converges.
\end{theorem}
\begin{proof}
    In the remainder of this section, we outline the main ideas of the proof.
The full argument is developed in Sections \ref{3.1}, \ref{sect:3.2}, and \ref{sect:3.3}.
\end{proof}

We emphasize that, although \eqref{eq:cont} provides an equivalent criterion for contractivity, it is not the contraction rate of the map
\[
\pmb \tau \mapsto \left((\tau_{0,i})_{i=1}^N,\left(H_{2,i}(Q^{\tau,i}_t, \tau_t)\right)_{\substack{i=1,\ldots,N\\ t=1,\ldots,T-1}}\right)
\]
in $\prod_{t=1}^{T-1}\left(\prod_{i=1}^N \mathcal P(X)\right)\times \prod_{i=1}^N\{\tau_{0,i}\}$ for any $(\tau_{0,i})\in \prod_{i=1}^N \mathcal P(X)$ under any norm. 
A particular case of Theorem \ref{thrm:quartic} is the single-population setting, in which we recover a contraction condition equivalent to that of \cite[Theorem 2]{ayd}.

\begin{corollary}
    Suppose $N=1$, i.e. we have a single population MFG. Then, we have contraction for the mean-field equilibrium operator if 
    \begin{equation}\label{eq:refff}
    \inf_{\overline K_1 -K_1 < r <\beta_1^{-1}}\frac{\overline K_1 +\frac{\overline L_1K_1}{\rho_1}}{r} + \frac{\overline L_1 K_1\beta_1}{\rho_1(1-\beta_1 r)}<1.
    \end{equation}
\end{corollary}
\begin{proof}
    This follows directly from Theorem \ref{thrm:quartic}.
\end{proof}

\begin{remark}
    A condition that is equivalent to \eqref{eq:refff} is given in \cite[Theorem 2]{ayd}. The main advantage of the corollary above compared to \cite[Theorem 2]{ayd} is that, we will be able to obtain explicit convergence rates from finite-horizon to infinite-horizon MFE under the variational characterization under slightly weaker conditions than the ones used in \cite{ayd}.
\end{remark}

The proof of Theorem \ref{thrm:quartic} relies on finding a uniform upper bound on the spectral radius of a matrix as done in Section \ref{sect:stat}. The main motivation behind this result is that minimizers of the left-hand side of \eqref{eq:cont} allow us to construct a Lyapunov function that yields explicit convergence rates between finite-horizon and infinite-horizon non-stationary MFEs, which is the subject of the next section. These convergence rates also allow us to deduce a uniqueness condition for infinite-horizon non-stationary MFEs.

As in the stationary case, we begin our analysis by establishing a vector inequality for the finite-horizon setting that describes the worst-case behavior of the iterations obtained under the map $\pmb \tau \mapsto \left((\tau_{0,i})_{i=1}^N,\left(H_{2,i}(Q^{\pmb\tau,i}_t,\tau_t)\right)_{i=1,t=1}^{N,T-1}\right)$.

\begin{lemma}\label{lem:4}
    For all $i \in [N]$, let $\bar {\mathcal M}_i, {\mathcal M}_i \in \mathbb R^{(T-1)\times (T-1)}$ be such that $\overline {\mathcal M}_i = T_i(\overline K_i)$ and $\mathcal M_i = T_i(K_i)$, where
   \[
   T_i(x) = \begin{bmatrix}
        \frac{\overline L_i K_i}{\rho_i}\beta_i &\frac{\overline L_i K_i}{\rho_i}\beta_i^2 & \frac{\overline L_i K_i}{\rho_i}\beta_i^3 & \cdots & \frac{\overline L_i K_i}{\rho_i}\beta_i^{T-1} \\
        x + \frac{\overline L_i K_i}{\rho_i} & \frac{\overline L_i K_i}{\rho_i}\beta_i & \frac{\overline L_i K_i}{\rho_i}\beta_i^2 & \cdots & \frac{\overline L_i K_i}{\rho_i}\beta_i^{T-2}\\
        0 & x + \frac{\overline L_i K_i}{\rho_i} & \frac{\overline L_i K_i}{\rho_i}\beta_i & \cdots & \frac{\overline L_i K_i}{\rho_i}\beta_i^{T-3} \\
        \vdots & \vdots & \vdots & \ddots& \vdots \\
        0 & 0 & 0 & \cdots & \frac{\overline L_i K_i}{\rho_i}\beta_i
    \end{bmatrix}.
    \]
    Let $E_{i,j}$ denote an $N \times N$ matrix with a $1$ in the $(i,j)$-entry and zeros elsewhere. We define
    \[
    \mathcal S_T = \sum_{i=1}^N E_{i,i} \otimes \overline {\mathcal M_i} +\sum_{i,j:i \not = j}^N E_{i,j} \otimes \mathcal M_i.
    \]
    Then, for any $\tau,\tilde \tau \in  \prod_{t=1}^{T-1}\left(\prod_{i=1}^N \mathcal P(X)\right)\times \prod_{i=1}^N\{\tau_{0,i}\}$ we have
    \[
    \begin{bmatrix}
        \|H_{2,1}(Q^{\tau,1}_0,\tau_0)-H_{2,1}(Q^{\tilde \tau,1}_0,\tilde \tau_0) \|_1\\
        \|H_{2,1}(Q^{\tau,1}_1,\tau_1)-H_{2,1}(Q^{\tilde \tau,1}_1,\tilde \tau_1) \|_1\\
        \vdots \\
        \|H_{2,1}(Q^{\tau,1}_{T-2},\tau_{T-2})-H_{2,1}(Q^{\tilde \tau,1}_{T-2},\tilde \tau_{T-2}) \|_1\\
        \|H_{2,2}(Q^{\tau,2}_0,\tau_0)-H_{2,2}(Q^{\tilde \tau,2}_0,\tilde \tau_0) \|_1\\
        \|H_{2,2}(Q^{\tau,2}_1,\tau_1)-H_{2,2}(Q^{\tilde \tau,2}_1,\tilde \tau_1) \|_1\\
        \vdots \\
        \|H_{2,2}(Q^{\tau,2}_{T-2},\tau_{T-2})-H_{2,2}(Q^{\tilde \tau,2}_{T-2},\tilde \tau_{T-2}) \|_1\\
        \vdots \\
        \|H_{2,N}(Q^{\tau,N}_0,\tau_0)-H_{2,N}(Q^{\tilde \tau,N}_0,\tilde \tau_0) \|_1\\
        \|H_{2,N}(Q^{\tau,N}_1,\tau_1)-H_{2,N}(Q^{\tilde \tau,N}_1,\tilde \tau_1) \|_1\\
        \vdots \\
        \|H_{2,N}(Q^{\tau,N}_{T-2},\tau_{T-2})-H_{2,N}(Q^{\tilde \tau,N}_{T-2},\tilde \tau_{T-2}) \|_1
    \end{bmatrix} \le \mathcal S_T \begin{bmatrix}
        \|\tau_{1,1}-\tilde \tau_{1,1}\|_1\\
        \|\tau_{2,1}-\tilde \tau_{2,1}\|_1\\
        \vdots \\
        \|\tau_{T-1,1}-\tilde \tau_{T-1,1}\|_1\\
        \| \tau_{1,2}- \tilde \tau_{1,2} \|_1\\
        \| \tau_{2,2}- \tilde \tau_{2,2} \|_1\\
        \vdots \\
        \| \tau_{T-1,2}- \tilde \tau_{T-1,2} \|_1\\
        \vdots\\
        \|\tau_{1,N}-\tilde \tau_{1,N}\|_1\\
        \|\tau_{2,N}-\tilde \tau_{2,N}\|_1\\
        \vdots \\
        \|\tau_{T-1,N}-\tilde \tau_{T-1,N}\|_1
    \end{bmatrix}.
    \]
\end{lemma}

\begin{proof}
    Fix an arbitrary $i \in [N]$. For any $0<t\le T-2$, using Lemma \ref{lem:1}, through recursive iterations, we obtain
    \begin{align*}
        &\| H_{2,i}(Q^{\tau,i}_t,\tau_t)-H_{2,i}(Q^{\tilde \tau,i}_t,\tilde \tau_t)\|_1 
        \\&\le \frac{K_i}{\rho_i}\|Q^{\tau,i}_t-\tilde Q^{\tilde \tau,i}_t\|_{\infty} + \overline K_i\|\tau_{t,i}-\tilde \tau_{t,i}\|_1 + K_i\sum_{j \not = t}\| \tau _{t,j} - \tilde \tau_{t,j}\|_1
        \\&\le \frac{K_i\beta_i}{\rho_i}\|Q^{\tau,i}_{t+1}-\tilde Q^{\tilde \tau,i}_{t+1}\|_{\infty} + \left( \overline K_i + \frac{\overline L_i K_i}{\rho_i}\right)\|\tau_{t,i}-\tilde \tau_{t,i}\|_1 + \left( K_i + \frac{\overline L_i K_i}{\rho_i}\right)\| \tau _{t,j} - \tilde \tau_{t,j}\|_1
        \\& \le \left( \overline K_i + \frac{\overline L_i K_i}{\rho_i}\right)\|\tau_{t,i}-\tilde \tau_{t,i}\|_1 + \left( K_i + \frac{\overline L_i K_i}{\rho_i}\right)\| \tau _{t,j} - \tilde \tau_{t,j}\|_1 
        \\& \qquad + \frac{\overline L_iK_i}{\rho_i}\sum_{k=1}^{T-1-t} \beta^k\sum_{j } \|\tau_{t+k,j}-\tilde \tau_{t+k,j}\|_1.
    \end{align*}
    For $t=0$, since $\tau_0=\tilde \tau_0$, repeating the same argument above, we obtain that
    \begin{align*}
        &\| H_{2,i}(Q^{\tau,i}_0,\tau_0)-H_{2,i}(Q^{\tilde \tau,i}_0,\tilde \tau_0)\|_1 \le \frac{\overline L_iK_i}{\rho_i}\sum_{k=1}^{T-1} \beta^k\sum_{j } \|\tau_{t+k,j}-\tilde \tau_{t+k,j})\|_1.
    \end{align*}
    These two expressions give the desired matrix inequality.
\end{proof}

{In the construction of $\mathcal S_T$, each block $E_{i,j}$ captures the impact of population $j$ on the dynamics of population $i$.}

Due to lack of sparsity of the matrix \(\mathcal S_T\), numerical estimation of its spectral radius \(\rho(\mathcal S_T)\) becomes more difficult as \(T\) increases. Moreover, the dimension of \(\mathcal S_T\) also depends on the number of populations, which further complicates the problem. The asymptotic behavior of the ratios of eigenvector components corresponding to \(\rho(\mathcal S_T)\) may also provide insight into finite-time error rates between finite-horizon and infinite-horizon MFEs; see \cite[Theorem 4]{ayd}. In general, we will not be able to provide a closed-form expression for \(\rho(\mathcal S_T)\) or for \(\lim_{T \to \infty} \rho(\mathcal S_T)\). Instead, we provide a contraction condition that can be computed efficiently numerically. In particular, to prove Theorem \ref{thrm:quartic}, we prove the following:
\begin{theorem}\label{thrm:4}
    We have $\lim_{T \to \infty} \rho(\mathcal S_T) <1$ if, and only if,
    \begin{equation}\label{eq:cont1}
    \inf_{ \overline K_\infty < r < \beta_{\max}^{-1}} \sum_{i \in [N]}\frac{\frac{K_i+\frac{\overline L_i K_i}{\rho_i}}{r}+\frac{\overline L_iK_i\beta_i}{\rho_i}\frac{1}{1-\beta_ir}}{1-\frac{\overline K_i-K_i}{r}} < 1.
    \end{equation}
\end{theorem}
\begin{proof}
    A proof for this result will be established through sections \ref{3.1}, \ref{sect:3.2}, and \ref{sect:3.3}.
\end{proof}

{Note that the sequence $(\rho(\mathcal S_T))_T$ is a monotone sequence. Thus, the contraction condition in Theorem \ref{thrm:4} is applicable to all $\rho(\mathcal S_T)$.} The difficult part of the proof of Theorem \ref{thrm:4} is showing that $\lim_{T \to \infty} \rho(\mathcal S_T) <1$ implies \eqref{eq:cont1}. Proving that \eqref{eq:cont1} implies  $\lim_{T \to \infty} \rho(\mathcal S_T) <1$ will be argued through a Collatz-Wiedelant argument by choosing geometric test vectors. If \(\lim_{T \to \infty} \rho(\mathcal S_T) < 1\), then part of the proof of Theorem \ref{thrm:4} implies that \(\max_j(\overline K_j-K_j) < r < \beta_{\max}^{-1}\). Thus, when \eqref{eq:cont1} is satisfied with \(r=1\), we recover the contraction condition for stationary MFEs; see Theorem \ref{thrm:2}. In this sense, the contraction condition for finite-horizon MFEs is strictly weaker than that for stationary MFEs.

Let
\[
r \in (\max_{j \in [N]}(\overline K_j-K_j), \beta^{-1}_{\max}) \mapsto V(r) = \sum_{i \in [N]}\frac{\frac{K_i+\frac{\overline L_i K_i}{\rho_i}}{r}+\frac{\overline L_iK_i\beta_i}{\rho_i}\frac{1}{1-\beta_ir}}{1-\frac{\overline K_i-K_i}{r}}.
\]
We call \(V\) a \emph{variational contraction function} for the finite-horizon MFE. The value \(\inf_r V(r)\) represents the \emph{worst-case} behavior of the finite-horizon MFG that can be inferred from a majorization of \(\mathcal S_T\). In the next section, when we derive finite-time error bounds between finite-horizon and infinite-horizon MFEs, the regime of the values of \(r\) such that \(V(r)<1\) will play a critical role. In general, we can separate the behavior of \(V\) into three regimes: \(r<1\), \(r=1\), and \(r>1\).

\begin{enumerate}
    \item When \(V(1)<1\), there exists a majorization of \(\mathcal S_T\) such that the ratios of consecutive entries of a positive Perron right eigenvector converge to the constant \(1\). Note that the case \(V(1)\) corresponds to the contraction rate in Theorem \ref{thrm:2}; thus, the asymptotic stationarity of these eigenvectors implies that the system approaches consensus. In this case, we say that \(\mathcal S_T\) is \emph{asymptotically stationary}.
    \item When \(V(r)<1\) for some \(r>1\), there exists a majorization of \(\mathcal S_T\) such that the ratios of consecutive entries of a positive Perron right eigenvector converge to the constant \(r\). In this case, we say that \(\mathcal S_T\) is \emph{stable}. This stability comes from the fact that the corresponding majorization has Perron eigenvectors that decay geometrically with ratio \(1/r\). Stable cases are those in which we can recover a convergence rate from finite-horizon MFEs to infinite-horizon MFEs. Note that, by continuity of \(V\), if \(\mathcal S_T\) is asymptotically stationary, then it is also stable.
    \item When \(V(r)<1\) only for \(r<1\), there exists a majorization of \(\mathcal S_T\) such that the ratios of consecutive entries of a positive Perron right eigenvector converge to the constant \(r\). In this case, we say that \(\mathcal S_T\) is \emph{unstable}.
\end{enumerate}

The results in \cite[Theorems 2 and 4]{ayd} are equivalent to the worst-case quantity \(\inf_r V(r)\) in the single-population setting. For instance, as we show in the next section, when the stationary MFG is eventually contractive, i.e., when \(V(1)<1\), we can still recover a convergence rate between finite-horizon and infinite-horizon MFEs using the variational contraction function \(V\). In this sense, the result above generalizes \cite[Theorem 2]{ayd}.

Our proof strategy for Theorem \ref{thrm:4} amounts to showing that appropriate majorizations of \(\mathcal S_T\) with the stated characteristics exist in each of the three regimes described above. These characteristics are identified by studying a limiting eigenvector equation derived from the finite-horizon Perron right eigenvectors of \(\mathcal S_T\). Note that for each \(T \in \mathbb N\), the Perron right eigenvector of \(\mathcal S_T\) satisfies
\begin{equation}\label{eq:BC}
\rho(\mathcal S_T)\,u_{i}(1) = \frac{\overline L_iK_i}{\rho_i} \sum_{j \in [N]}\sum_{k=1}^{T-1}\beta_i^k\,u_j(k),  \qquad i\in [N],
\end{equation}
and
\begin{equation}\label{eq:BULK}
\rho(\mathcal S_T)\,u_{i}(k) = (\overline K_i-K_i)u_i(k-1) + \sum_{j \in [N]}K_i\,u_j(k-1) +\frac{\overline L_iK_i}{\rho_i}\sum_{j \in [N]}\sum_{t=k-1}^{T-1} \beta_i^{t-k+1}\,u_j(t),
\end{equation}
for \(i\in [N]\) and \(2 \le k \le T-1\). In what follows, we let \(\lambda_T:= \rho(\mathcal S_T)\). 

Let \(\lambda_{\infty} = \lim_{T \to \infty} \lambda _T\), which exists since the sequence \((\lambda_T)_{T\ge0}\) is monotone. On the space of real sequences \((\mathbb R^{\mathbb N})^{N}\), for \(u=(u_i(k))_{k \in \mathbb N,\, i \in [N]}\) and any \(\lambda>0\), we define the \emph{limiting eigenvector operator} \(\mathcal T_{\lambda}\) by taking \(T \to \infty\) in the relations above:
\begin{equation}\label{eq:boundary}
(\mathcal T_{\lambda}(u))_i(1) = \frac{\overline L_iK_i}{\rho_i\lambda}\sum_{k \ge 1}\beta_i^k\sum_{j \in [N]} u_j(k), \,\, i \in [N]
\end{equation}
and for $k\ge 2$
\begin{equation}\label{eq:limiteig}
(\mathcal T_{\lambda}(u))_i(k) = \frac{(\overline K_i- K_i)}{\lambda}u_i(k-1) +\frac{K_i}{\lambda}\sum_{j \in [N]}u_j(k-1) + \frac{\overline L_iK_i}{\rho_i\lambda} \sum_{m \ge 0}\beta_i^m\sum_{j\in [N]}u_j(k+m-1), \,\, i \in [N].
\end{equation}
Both in the finite-horizon and infinite-horizon cases, the conditions satisfied under $k=1$ (i.e., \eqref{eq:boundary}) will be referred to as \emph{boundary conditions}.

Informally, we will show that if \(\mathcal T_{\lambda_\infty}\) has a fixed point \(u \in (\mathbb R^{\mathbb N})^{N}\) with strictly positive components, then for all \(i \in [N]\) we have
\[
\lim_{k \to \infty} \frac{u_i(k+1)}{u_i(k)} = \text{constant},
\]
and this constant is the same for all \(i\in [N]\). This limit corresponds to a parameter \(r^* \in (\max_j(\overline K_j-K_j), \beta^{-1}_{\max})\) satisfying \(\inf_r V(r) = V(r^*)\). The value(s) of \(r^*\) that attain \(V(r^*)\) will be crucial for deriving finite-time error bounds between finite-horizon and infinite-horizon MFEs, which is the subject of the next section. To prove Theorem \ref{thrm:4}, we will need the following necessary quantitative criterion obtain for the existence of a positive fixed-point of \(\mathcal T_{\lambda}\), which also might be of independent interest:

\begin{theorem}\label{thrm:5}
Fix \(\lambda >0\). Suppose that \(\min_j \lambda/(\overline K_j - K_j)>1\). If the operator \(\mathcal T_{\lambda}\) has a positive fixed point then
\begin{equation}\label{eq:constraint}
1 = \sum_{i \in [N]}\frac{ K_i z_0+\frac{a_iz_0^2}{z_0-\beta_i}}{\lambda-(\overline K_i-K_i)z_0}
\end{equation}
for some \(z_0 \in \bigl(\beta_{\max}, \min_j\lambda/(\overline K_j-K_j)\bigr)\).
\end{theorem}
\begin{proof}
    A proof for this will be established through sections \ref{3.1}, \ref{sect:3.2}, and \ref{sect:3.3}.
\end{proof}

When $N=1$, the theorem above provides a criterion for the existence of a positive fixed-point of a one-sided infinite Toeplitz operator. In general, infinite dimensional Toeplitz operators are known to be non-compact; thus, their spectrum cannot be approximated by finite truncations. A particular implication of \cite[Proposition 3]{ayd} in the single population case is that one does not have $\lim_{T \to \infty}\rho(\mathcal S_T) = \rho(\mathcal T_{\lambda_{\infty}})$ when $\mathcal T_{\lambda_{\infty}}$ is defined on the space of bounded sequences $(\ell_{\infty},\| \cdot\|_{\infty})$, where $\| \cdot \|_{\infty}$ is the uniform norm. Part of the proof of Theorem \ref{thrm:5} relies on identification of a weighted space where we have $\lim_{T \to \infty}\rho(\mathcal S_T) = \rho(\mathcal T_{\lambda_{\infty}})$.

For our purposes, the proof of Theorem \ref{thrm:4} has three main components:
\begin{enumerate}
    \item The operator \(\mathcal T_{\lambda_\infty}\) has a positive fixed point in \((\mathbb R^{\mathbb N})^N\) (Proposition \ref{prop:0}).
    \item Let \(\lambda>0\). If \(\mathcal T_{\lambda}\) has a positive fixed point \(u \in (\mathbb R^{\mathbb N})^N\), then for its fixed point $(u_i(k))_{i,k}$, \(\lim_{k \to \infty} \frac{u_i(k+1)}{u_i(k)}\) exists and equals a constant independent of \(i \in [N]\) (Lemma \ref{lem:54}).
    \item If \(\mathcal T_{\lambda}\) has a positive fixed point and \(\lambda > \max_j(\overline K_j-K_j)\), then \(\lambda = \rho(B(r))\) for some $r$ (Lemmas \ref{lem:55} and \ref{lem:unique-e}), where \(\rho(B(r))\) denotes the spectral radius of the matrix \(B(r)\) defined by
    \[
    B(r) := D(r) +  v(r)\,\mathbf 1^{\top}.
    \]
    Here, \(\mathbf 1\) is the vector of ones and
    \[
    v_j(r) = \frac{K_j+\frac{\overline L_jK_j}{\rho_i}}{r}+\frac{\overline L_j K_j\beta_j}{\rho_j\left(1-\beta_jr\right)}, \qquad
    D(r):=\mathrm{diag}\left( \frac{\overline K_1 - K_1}{r},\frac{\overline K_2-K_2}{r},\cdots ,\frac{\overline K_N-K_N}{r}\right),
    \]
    \[
    v(r) :=(v_1(r),\cdots,v_N(r))^{\top}.
    \]
    The spectral radius of \(B(r)\) provides a majorization of the spectral radius of \(\mathcal S_T\) for all \(T\). Using multiple copies of \(B(r)\), we can also construct a matrix that directly majorizes \(\mathcal S_T\); this will be done in the next section when deriving convergence rates from finite-horizon to infinite-horizon MFEs. The constraint \eqref{eq:constraint} then will be shown to hold for $\rho(B(r))$.
\end{enumerate}

The first item will be proved using the dominated convergence theorem by showing the uniform integrability of Perron eigenvectors of $\mathcal S_T$ over the counting measure. To show this, we will heavily rely on the geometric decaying tails of the matrices $\mathcal S_T$ and work with normalized Perron eigenvectors (Proposition \ref{prop:0}), which establishes an asymptotic regularity property for these Perron eigenvectors.

The second item will be proved by mapping any positive fixed point of \(\mathcal T_{\lambda}\) to the complex plane using generating functions, inspired by methods from analytic combinatorics. As in the argument establishing regularity for the Perron right eigenvectors of \(\mathcal S_T\), we will show that these generating functions admit a sufficiently large radius of convergence around the origin $z=0$ (Lemmas \ref{lem:14} and \ref{lem:16}). We then use basic tools from singularity analysis in complex analysis to conclude the result (Lemma \ref{lem:54}).

The first two items above then will be sufficient to provide a sharp upper bound on \(\rho(\mathcal S_T)\) that is independent of \(T\). The last item (combined with the second item) is needed to obtain the characterization in Theorem \ref{thrm:4} using Theorem \ref{thrm:5}. An important feature of this step is that it avoids computing \(\inf_{0<r < \beta_{\max}^{-1}}\rho(B(r))\) as an intermediate step. Nevertheless, if one is interested in the asymptotics of \((\rho(\mathcal S_T))_T\), a useful application of the third item is the following limit result for the asymptotic behavior of \((\rho(\mathcal S_T))_T\).
\begin{proposition}\label{prop:quant}
    We have
    \[
    \lim_{T \to \infty}\rho(\mathcal S_T) = \inf_{0<r < \beta_{\max}^{-1}}\rho(B(r)).
    \]
\end{proposition}
\begin{proof}
    We postpone the proof to Section \ref{sect:3.3}
\end{proof}
The advantage of the calculation above, compared with computing \(\rho(\mathcal S_T)\) directly, is that the matrix \(B(r)\) is independent of the horizon length \(T\). It can also be shown that \(B(r)\) can be reduced to depend only on distinct population types, i.e., we can collapse identical populations into a single population.

\subsection{Step $1$: Upper Bound $\lambda_{\infty} \le \inf_{0<r<\beta^{-1}_{\max}}\rho(B(r))$}\label{3.1}

In this subsection, we show that the bound
\[
\lambda_{\infty} \le \inf_{0<r<\beta^{-1}_{\max}}\rho(B(r)) <\infty
\]
holds, and that for any \(0<r< \beta^{-1}_{\max}\),
\[
1 = \sum_{i \in [N]}\frac{\frac{K_i+\frac{\overline L_i K_i}{\rho_i}}{r}+\frac{\overline L_iK_i\beta_i}{\rho_i}\frac{1}{1-\beta_ir}}{\rho(B(r))-\frac{\overline K_i-K_i}{r}}.
\]

To obtain the upper bound \(\lambda_{\infty} \le \inf_{0<r<\beta^{-1}_{\max}}\rho(B(r))\), we use the Collatz--Wielandt inequalities with geometric test vectors. We will later show that this geometric ansatz yields a sharp upper bound on \(\rho(\mathcal S_T)\) as \(T\to \infty\), as outlined in the previous subsection. Moreover, the majorization of \(\lambda_{\infty}\) by \(\rho(B(r))\) will also allow us to derive convergence rates from finite-horizon to infinite-horizon MFEs.

We recall the Collatz-Wiedelant bound for completeness.

\begin{lemma}[Collatz-Wielandt bound]\label{lem:7}
    For any nonnegative square matrix $A$ and any positive vector $x=(x_i)_i>0$ (of appropriate dimension), it holds that
\[
\rho(A)\ \le\ \max_i\frac{(Ax)_i}{x_i}.
\]
\end{lemma}
\begin{proof}
    This follows from \cite[Theorem 8.3.2]{horn2012matrix}.
\end{proof}

\begin{lemma}\label{lem:8}
For every $T\ge 3$, we have
\[
\lambda_T \le \inf _{r\in (0, \beta_{\max}^{-1})}\rho(B(r)),
\]
where $\rho(B(r))$ is the spectral radius of the matrix $B(r)$.
\end{lemma}

\begin{proof}
We first construct a positive
test vector to which we will apply the Collatz--Wielandt bound. Fix $r\in(0,\beta_{\max}^{-1})$. 
Define $x=x(r)\in\mathbb R_{>0}^{T-1}$ by
$x(r)=(r^{i-1})_{i=1}^{T-1}.$ 
For $c=(c_1,\ldots,c_N)\in\mathbb R_{>0}^N$, define
\[
\bar x=\bar x(r,c)
:=
(c_1x,c_2x,\ldots,c_Nx)
\in\mathbb R_{>0}^{N(T-1)},
\]
and let $C:=\sum_{i\in[N]}c_i.$

We next estimate the Collatz--Wielandt ratios corresponding to
this choice of $\bar x$. For every $m\in[N]$ and
$i\in\{2,\ldots,T-1\}$, the definition of $\mathcal S_T$ gives
\[
(\mathcal S_T\bar x)_{(m-1)(T-1)+i}
=
(\overline K_m-K_m)c_mr^{i-2}
+
C\left(
K_m+\frac{\overline L_mK_m}{\rho_m}
\right)r^{i-2}
+
\frac{C\overline L_mK_m}{\rho_m}
\sum_{j=i}^{T-1}
\beta_m^{j-i+1}r^{j-1}.
\]
Hence,
\begin{align*}
\frac{(\mathcal S_T\bar x)_{(m-1)(T-1)+i}}{c_mx_i}
&=
\frac{\overline K_m-K_m}{r}
+
\frac{
C\left(K_m+\frac{\overline L_mK_m}{\rho_m}\right)
}{c_mr}
+
\frac{C\overline L_mK_m}{\rho_mc_m}
\sum_{j=i}^{T-1}
\beta_m^{j-i+1}r^{j-i}
\\
&\le
\frac{\overline K_m-K_m}{r}
+
\frac{
C\left(K_m+\frac{\overline L_mK_m}{\rho_m}\right)
}{c_mr}
+
\frac{C\overline L_mK_m}{\rho_mc_m}
\sum_{j=0}^{\infty}\beta_m^{j+1}r^j
\\
&=
\frac{\overline K_m-K_m}{r}
+
\frac{
C\left(K_m+\frac{\overline L_mK_m}{\rho_m}\right)
}{c_mr}
+
\frac{C\overline L_mK_m\beta_m}
{\rho_mc_m(1-\beta_mr)}
\\
&=
\frac{(B(r)c)_m}{c_m},
\end{align*}
where the infinite series is finite since
$r<\beta_{\max}^{-1}\le\beta_m^{-1}$.
The same also holds when $i=1$. Therefore,
\[
\max_{\substack{m\in[N]\\1\le i\le T-1}}
\frac{(\mathcal S_T\bar x)_{(m-1)(T-1)+i}}{c_mx_i}
\le
\max_{m\in[N]}
\frac{(B(r)c)_m}{c_m}.
\]
We now optimize over the population weights $c$. By the
Collatz--Wielandt characterization of the spectral radius,
\[
\inf_{c \in \mathbb R_{>0}^N}
\max_{m\in[N]}
\frac{(B(r)c)_m}{c_m}
=
\rho(B(r)).
\]
On the other hand, applying the Collatz--Wielandt bound to
$\mathcal S_T$ with the positive vector $\bar x$ yields
\[
\lambda_T
=
\rho(\mathcal S_T)
\le
\max_{\substack{m\in[N]\\1\le i\le T-1}}
\frac{(\mathcal S_T\bar x)_{(m-1)(T-1)+i}}{c_mx_i}.
\]
Combining the preceding inequalities and taking the infimum over $c>0$ gives
\[
\lambda_T
\le
\inf_{0<r<\beta_{\max}^{-1}}\rho(B(r)).
\]

It remains to pass to the limit as $T\to\infty$. The matrix
$\mathcal S_T$ can be identified with a principal submatrix of
$\mathcal S_{T+1}$. Since both matrices are nonnegative, monotonicity
of the spectral radius with respect to principal submatrices implies $\lambda_T\le\lambda_{T+1}.$
Thus, $(\lambda_T)_{T\ge3}$ is nondecreasing. The uniform bound
established above shows that it is also bounded from above, and hence $\lambda_\infty:=\lim_{T\to\infty}\lambda_T$
exists. Passing to the limit in $\lambda_T\le\inf_{0<r<\beta_{\max}^{-1}}\rho(B(r))$
yields $\lambda_\infty\le\inf_{0<r<\beta_{\max}^{-1}}\rho(B(r)),$ as desired.
\end{proof}

Lemma \ref{lem:8} also serves as the first half of the proof of Proposition \ref{prop:quant} as we have shown that $\lim_{T \to \infty}\rho(\mathcal S_T) \le \inf_{0<r < \beta^{-1}_{\max}}\rho(B(r)).$

Next, we establish a relation between \(\rho(B(r))\) and the parameter \(r\). This will later be used to relate the values of \(\lambda>0\) for which \(\mathcal T_{\lambda}\) admits a positive fixed point to corresponding values of \(\rho(B(r))\).

\begin{lemma}\label{lem:57}
For any \(0<r < \beta^{-1}_{\max}\), we have
\[
1 = \sum_{i \in [N]}\frac{\frac{K_i+\frac{\overline L_i K_i}{\rho_i}}{r}+\frac{\overline L_iK_i\beta_i}{\rho_i}\frac{1}{1-\beta_ir}}{\rho(B(r))-\frac{\overline K_i-K_i}{r}}.
\]
Furthermore, it holds that
\[
\beta^{-1}_{\max} > r > \max_{j\in[N]} \frac{\overline K_j-K_j}{\rho(B(r))}.
\]
\end{lemma}
\begin{proof}
    The proof closely follows that of Theorem \ref{thrm:2}. Since $B(r)$ is a nonnegative irreducible matrix, by the Perron-Frobenius theorem, it holds that there exists a nonnegative vector $\bar c=(c_1,c_2,\cdots,c_N)$ such that $B(r)\bar c = \rho(B(r))\bar c.$ Then, for any $m \in [N]$, we have
    \[
    (\overline K_m-K_m)\frac{c_m}r + k =\rho(B(r))c_m \implies k = \left( \rho(B(r))-(\overline K_m-K_m) \frac 1r\right)c_m,
    \]
    where 
    \begin{equation}
    k = \sum_{i \in [N]} \left( \frac{K_i+\frac{\overline L_i K_i}{\rho_i}}r + \frac{\overline L_i K_i\beta_i}{\rho_i(1-\beta_ir)}\right)c_i.
    \end{equation}
    Thus,
    \begin{equation}\label{eq:hh}
    k = \sum_{i \in [N]} \left( \frac{K_i +\frac{\overline L_i K_i}{\rho_i}}r + \frac{\overline L_i K_i\beta_i}{\rho_i(1-\beta_ir)}\right)c_i =\sum_{i \in [N]} \frac{\left( \frac{K_i+\frac{\overline L_i K_i}{\rho_i}}r + \frac{\overline L_i K_i \beta_i}{\rho_i(1-\beta_ir)}\right)}{\rho(B(r))-\frac{\overline K_i - K_i}{r}}k.
    \end{equation}
    Note that $\rho(B(r)) > \frac{\overline K_i-K_i}{r}$ for all $i \in [N]$ due to monotonicity of spectral radius with respect to diagonal terms; so, it holds that 
    \[
    \frac{\left( \frac{K_i+\frac{\overline L_i K_i}{\rho_i}}r + \frac{\overline L_i K_i \beta_i}{\rho_i(1-\beta_ir)}\right)}{\rho(B(r))-\frac{\overline K_i - K_i}{r}}  > 0.
    \]
    Thus, \eqref{eq:hh} holds if and only if 
    \[
    \sum_{i \in [N]} \frac{\left( \frac{K_i+\frac{\overline L_i K_i}{\rho_i}}r + \frac{\overline L_i K_i \beta_i}{\rho_i(1-\beta_ir)}\right)}{\rho(B(r))-\frac{\overline K_i - K_i}{r}} = 1,
    \]
    as desired.
\end{proof}

\subsection{Step 2:Existence of a positive fixed point of $\mathcal T_{\lambda_\infty}$}\label{sect:3.2}

In this subsection, we will show that for $\lambda_{\infty}$, there exists a positive fixed point of the operator $\mathcal T_{\lambda_{\infty}}$. To show this, as discussed before, we will use the right Perron eigenvectors of $\mathcal S_T$ under a particular normalization. Using the structure of the matrices $\mathcal S_T$, we will show that these Perron eigenvectors admit sufficient regularity conditions to use the dominated convergence theorem. 

First, we show that, under a suitable normalization, the Perron right eigenvectors of \(\mathcal S_T\) admit a convergent subsequence as \(T \to \infty\). We choose a normalization that fixes the first coordinate of the population \(1\). Accordingly, we relate the first coordinates of the other populations to that of population \(1\) using the boundary conditions of the operator $\mathcal T_{\lambda_{\infty}}$.

\begin{lemma}[Finite-horizon geometric lower bounds]\label{lem:lower_finite}
Fix $T\ge 2$ and let $\lambda_T>0$ be the Perron eigenvalue of $\mathcal S_T$.
Let $u_i^{(T)}(k)>0$ denote a positive right Perron eigenvectors of $\mathcal S_T$,
and define the aggregate $S^{(T)}(k):=\sum_{j=1}^N u_j^{(T)}(k)$ for $k=1,\dots,T-1$.
We normalize eigenvectors by setting $u_1^{(T)}(1)=1$ for all $T \in \mathbb N$.

Then, for every $i\in [N]$ and every $k=2,\dots,T-1$,
\begin{equation}\label{eq:lower_ratio_basic}
u_i^{(T)}(k)\ \ge\ \frac{\overline K_i-K_i}{\lambda_T}\,u_i^{(T)}(k-1).
\end{equation}
Consequently, for all $k=1,\dots,T-1$,
\begin{equation}\label{eq:lower_geom_basic}
u_i^{(T)}(k)\ \ge \left(\frac{\overline K_i-K_i}{\lambda_{\infty}}\right)^{k-1} \frac{\overline L K_i}{\rho_i\lambda_{\infty}}.
\end{equation}
Moreover, we have the upper bound
\[
\left(\frac{\rho_i\lambda_{\infty}}{\overline L_iK_i\beta^2_i} \right)^{k-1}\frac{\rho_1\lambda_{\infty}}{\overline L_1K_1\beta_1}   \ge u^{(T)}_i(k).
\]
\end{lemma}

\begin{proof}
First, we would like to set uniform upper and lower bounds for the terms $u^{(T)}_i(1)$. Since we set $u^{(T)}_1(1)=1$, note that the boundary conditions give
\[
u^{(T)}_i(1) \ge \frac{\overline L_iK_i}{\rho_i\lambda_T}\beta_iu_1^{(T)}(1) \ge \frac{\overline L_iK_i}{\rho_i\lambda_{\infty}}\beta_iu^{(T)}_1(1)=\frac{\overline L_iK_i}{\rho_i\lambda_{\infty}}\beta_i.
\]
Similarly, we must also have that
\[
\lambda_T = \lambda_T u^{(T)}_1(1) \ge \frac{\overline L_1K_1}{\rho_1}\beta_1 u^{(T)}_i(1) \implies \frac{\rho_1\lambda_{\infty}}{\overline L_1K_1\beta_1}\ge u^{(T)}_i(1).
\]

For $k=2,\dots,T-1$, the finite-horizon reduced Perron equations have the form
\begin{equation}\label{eq:finite_rec}
\lambda_T u_i^{(T)}(k)
=
\left(K_i+\frac{\overline L_iK_i}{\rho_i}\right)S^{(T)}(k-1)
+(\overline K_i-K_i) u_i^{(T)}(k-1)
+\frac{\overline L_iK_i}{\rho_i}\sum_{m=1}^{T-k}\beta_i^{m}S^{(T)}(k+m-1),
\end{equation}
where all coefficients and all terms on the right-hand side are nonnegative.
Dropping the nonnegative terms $\left(K_i+\frac{\overline L_iK_i}{\rho_i}\right)S^{(T)}(k-1)$ and the future-tail sum yields
\[
\lambda_T u_i^{(T)}(k) \ge \overline K_i-K_i u_i^{(T)}(k-1),
\]
which proves \eqref{eq:lower_ratio_basic}. Iterating from $2$ up to $k$ gives \eqref{eq:lower_geom_basic}:
\[
u_i^{(T)}(k)\ \ge\ \left(\frac{\overline K_i-K_i}{\lambda_T}\right)^{k-1} u_i^{(T)}(1) \ge \left(\frac{\overline K_i-K_i}{\lambda_{\infty}}\right)^{k-1} \frac{\overline L_iK_i}{\rho_i\lambda_{\infty}}.
\]

For the upper bound, using the boundary case of the eigenvector equation, i.e. when $k=1$, we obtain
\[
\lambda_T u^{(T)}_i(1) \ge \frac{\overline L_iK_i}{\rho_i}\beta^2_i u^{(T)}_i(2). 
\]
For $k\ge 2$, we now have that
\[
\lambda_T u_i^{(T)}(k) \ge \frac{\overline L_iK_i}{\rho_i}\beta^2_iu^{(T)}_i(k+1).
\]
Thus, for each \(k\), the uniform upper bound
\[
\left(\frac{\rho_i\lambda_{\infty}}{\overline L_iK_i\beta^2_i} \right)^{k-1}\frac{\rho_1\lambda_{\infty}}{\overline L_1K_1\beta_1} \ge \left(\frac{\rho_i\lambda_{T}}{\overline L_iK_i\beta^2_i} \right)^{k-1}u_i^{(T)}(1) \ge u^{(T)}_i(k)
\]
holds, as desired.
\end{proof}

Let \((u^{(T)}_i(k))_{i,k}\) be a positive Perron right eigenvector of \(\mathcal S_T\) normalized so that \(u^{(T)}_1(1)=1\). As a consequence of Lemma \ref{lem:lower_finite}, and using a diagonal argument, we can extract a subsequence \((T_n)_{n \in \mathbb N} \subset \mathbb N\) such that for all \(i \in [N]\) and \(k \in \mathbb N\),
\[
\lim_{n \to \infty} u^{(T_n)}_i(k) =: u_i(k) <\infty
\]
for some family of sequences \((u_i(k))_{i,k}\). In particular, \(u_i(k)>0\) for all \(k \ge 1\) and \(i \in [N]\). To show that \((u_i(k))_{i,k}\) is a positive fixed point of the limiting eigenvector equation \(\mathcal T_{\lambda_\infty}\), it suffices to justify the interchange of limit and summation:
\[
\lim_{n \to \infty} \sum_{k=1}^{T_n-1}\beta^k_i\sum_{i\in[N]}u^{(T_n)}_i(k)
=
\sum_{k \ge 1}\beta^k_i\sum_{i\in[N]}u_i(k).
\]
In the next lemma, we justify this exchange of limit and summation by showing that finite-horizon normalized right Perron eigenvectors are uniformly integrable under the geometric weights $\beta^k_i$; this will be the key step in obtaining a positive fixed point for \(\mathcal T_{\lambda_\infty}\).

\begin{lemma}\label{lem:beta_weighted_sum_conv}
Fix \(i\in [N]\) with \(\beta_i\in(0,1)\). For each \(T\ge 3\), let \(u_i^{(T)}(k)>0\) (\(k=1,\dots,T-1\)) be the Perron right eigenvector of \(\mathcal S_T\) corresponding to group \(i\), normalized by
\[
u_1^{(T)}(1)=1,
\]
and extend it by setting \(u_i^{(T)}(k)=0\) for all \(k\ge T\). Assume that \(\lambda_T\to\lambda_\infty\in(0,\infty)\) and that, for each fixed \(k\ge 1\), the limits
\[
u_i^{(T)}(k)\ \longrightarrow\ u_i(k)\qquad\text{as }T\to\infty
\]
exist.
Then,
\[
\lim_{T \to \infty} \sum_{k\ge 1}\beta_i^k\,u_i^{(T)}(k) = \sum_{k\ge 1}\beta_i^k\,u_i(k).
\]
\end{lemma}

\begin{proof}
Define the truncated tail sums
\[
S_{k,T}:=\sum_{m=k}^{T-1}\beta_i^{m-1}u_i^{(T)}(m),\qquad k=1,\dots,T-1.
\]
From the finite-horizon tail inequality (obtained by dropping nonnegative terms in the
$k$th eigenvector equation),
\begin{equation}\label{eq:tail}
u_i^{(T)}(k) \ge \frac{\overline L_iK_i}{\rho_i\lambda_T}\sum_{j=1}^{T-1-k}\beta_i^j\,u_i^{(T)}(k+j),
\qquad k=1,\dots,T-2.
\end{equation}
 Multiplying by $\beta_i^{k-1}$ yields
\[
S_{k+1,T} \le \frac{\rho_i\lambda_T}{\overline L_iK_i}\,\beta_i^{k-1}u_i^{(T)}(k),
\qquad k=1,\dots,T-2.
\]
Since $S_{k,T}=\beta_i^{k-1}u_i^{(T)}(k)+S_{k+1,T}$, we have
$\beta_i^{k-1}u_i^{(T)}(k)=S_{k,T}-S_{k+1,T}$ and therefore
\[
S_{k+1,T}\le \frac{\rho_i\lambda_T}{\overline L_iK_i}(S_{k,T}-S_{k+1,T})
\quad\Longrightarrow\quad
S_{k+1,T}\le \frac{\frac{\rho_i\lambda_T}{\overline L_iK_i}}{1+\frac{\rho_i\lambda_T}{\overline L_iK_i}}\,S_{k,T}.
\]
Iterating from $2$ to $k$ gives
\[
S_{k,T}\le S_{2,T}\left(\frac{\frac{\rho_i\lambda_T}{\overline L_iK_i}}{1+\frac{\rho_i\lambda_T}{\overline L_iK_i}}\right)^{k-2},
\qquad k=2,\dots,T-1.
\]
Because $\beta_i^{k-1}u_i^{(T)}(k)\le S_{k,T}$, multiplying by $\beta_i$ implies the key estimate:
\begin{equation}\label{eq:dominate_beta_u}
\beta_i^k u_i^{(T)}(k)\ \le\ S_{2,T}\left(\frac{\frac{\rho_i\lambda_T}{\overline L_iK_i}}{1+\frac{\rho_i\lambda_T}{\overline L_iK_i}}\right)^{k-2},
\qquad k=2,\dots,T-1.
\end{equation}

Next, under the normalization $u^{(T)}_1(1)$ for all $i\in [N]$ and $T \in \mathbb N$, apply \eqref{eq:tail} at $k=1$:
\[
\frac{\rho_1\lambda_{\infty}}{\overline L_1K_1\beta_1}u_i^{(T)}(1) \ge u_i^{(T)}(1)\ \ge\ \frac{\overline L_iK_i}{\rho_i\lambda_T}\sum_{m=2}^{T-1}\beta_i^{m-1}u_i^{(T)}(m)
=\frac{\overline L_iK_i\beta_i}{\rho_i\lambda_T}\,S_{2,T},
\]
so $S_{2,T}\le \frac{\rho_i\lambda_T}{\overline L_iK_i}$. Substituting into \eqref{eq:dominate_beta_u} yields
\[
\beta_i^k u_i^{(T)}(k) \le \frac{\rho_i\lambda_T}{\overline L_iK_i}\frac{\rho_1\lambda_{\infty}}{\overline L_1K_1\beta_1}\left(\frac{\frac{\rho_i\lambda_T}{\overline L_iK_i}}{1+\frac{\rho_i\lambda_T}{\overline L_iK_i}}\right)^{k-2},
\qquad k\ge 2.
\]
Since $\lambda_T\to\lambda_\infty$, and $(\lambda_T)_T$ is increasing, for all $T$ it holds that $\lambda_T\le 2\lambda_\infty$, hence
\[
\frac{\rho_i\lambda_T}{\overline L_iK_i}=\frac{\rho_i\lambda_T}{\overline L_iK_i}\le \frac{\rho_i2\lambda_\infty}{\overline L_iK_i}=:C_i,
\qquad
\frac{\frac{\rho_i\lambda_T}{\overline L_iK_i}}{1+\frac{\rho_i\lambda_T}{\overline L_iK_i}}
=\frac{\lambda_T}{\lambda_T+\frac{\overline L_iK_i}{\rho_i}}
\le \frac{2\lambda_\infty}{2\lambda_\infty+\frac{\overline L_iK_i}{\rho_i}}=:q_i<1.
\]
Therefore, for all $T\ge 0$ and all $k\ge 2$,
\[
0\le \beta_i^k u_i^{(T)}(k)\ \le\ C_i\frac{\rho_1\lambda_{\infty}}{\overline L_1K_1\beta_1}q_i^{k-2},
\]
which gives $T$-independent uniform bound on the Perron eigenvectors. 
Since $q_i<1$, and for all $C_i$ there exists some $K>0$ such that $C_i \le K$ for all \(i\), the majorant $C_i \frac{\lambda_{\infty}}{a_1\beta_1}q_i^{k-2}$ is summable over $k\ge 2$, and by assumption
$\beta_i^k u_i^{(T)}(k)\to \beta_i^k u_i(k)$ pointwise for each fixed $k$.

The upper bounds we have established above imply that the family $(\beta^k_iu^{T}_i(k))_{i,k}$ is uniformly integrable under the counting measure. Hence, the dominated convergence theorem applied to the counting measure gives
\[
\lim_{T \to \infty} \sum_{k\ge 2}\beta_i^k u_i^{(T)}(k) = \sum_{k\ge 2}\beta_i^k u_i(k).
\]
Finally, adding the $k=1$ terms:
$\beta_i u_i^{(T)}(1)\to \beta_i u_i(1)$, yields the limit
\[
\lim_{T \to \infty} \sum_{k\ge 1}\beta_i^k u_i^{(T)}(k) = \sum_{k\ge 1}\beta_i^k u_i(k),
\]
as claimed.
\end{proof}

In the next proposition, we state formally the existence of a positive fixed point of \(\mathcal T_{\lambda_{\infty}}\), which we will reference later.

\begin{proposition}\label{prop:0}
Let \(\mathfrak u^T=(u^{(T)}_i(k))_{i,k}\) be a positive Perron right eigenvector of \(\mathcal S_T\), and normalize it by setting \(u^{(T)}_1(1)=1\). Then any accumulation point of the sequence $\mathfrak u^T$ in the space of sequences $\mathbb R^{\mathbb N}$ (endowed with the product topology) as \(T \to \infty\) is a fixed point of \(\mathcal T_{\lambda_\infty}\). In particular, \(\mathcal T_{\lambda_\infty}\) has a positive fixed point.
\end{proposition}
\begin{proof}
    Note that accumulation points of the sections {of $\mathfrak u^T$} are finite by Lemma \ref{lem:lower_finite}. {Furthermore, each of these sections lie in a compact subset of $\mathbb R$ in view of which an accumulation point exists for each section.} Thus, this result follows after writing the eigenvector equations for $(u^{(T)}_i(k))_{i,k}$ and applying Lemma \ref{lem:beta_weighted_sum_conv} to exchange the limit and the summation. Since each section of $(u^{(T)}_i(k))_{i,k}$ is positive and bounded away from $0$ by Lemma \ref{lem:lower_finite}, it follows that the accumulation points of $(u^{(T)}_i(k))_{i,k}$ are also {(finite)} positive sequences, and hence, $\mathcal T_{\lambda_\infty}$ admits a positive fixed point. 
\end{proof}

\subsection{Step 3: Limiting Behavior of the Ratio of Consecutive Terms of  Perron Eigenvectors}\label{sect:3.3}

In this subsection, we present the final step in the proof of Theorem \ref{thrm:quartic} by filling in the missing details from the previous section. In particular, we prove that the fixed points of $\mathcal T_{\lambda}$ satisfying the limiting eigenvector equation (corresponding to \(T=\infty\)) imply that the ratios of consecutive components converge to the same limit. To this end, we use ideas from analytic combinatorics and translate the limiting solutions to the limiting eigenvector equation to the complex plane via generating functions associated with the eigenvectors and the ``bulk'' term. We then apply singularity analysis to show that the ratios of consecutive eigenvector components converge to a common constant.

The main difficulty in this step is to establish an appropriate radius of convergence for these generating functions. An initial estimate for the radius of convergence will be obtained by showing that positive fixed points of \(\mathcal T_{\lambda}\) satisfy sufficient regularity conditions, as in Lemma \ref{lem:beta_weighted_sum_conv}. We then obtain the desired radius of convergence by constructing suitable analytic extensions of these generating functions after a series of algebraic manipulations. 

To this end, we decompose the equations induced by \(\mathcal T_{\lambda}\) into two parts: the eigenvector component \(u_i(k)\) and the ``bulk'' term \(s(k)\), defined as the sum of eigenvector components at time \(k\). We first perform singularity analysis on the generating function of \(s\). After deriving a relationship between \(s\) and \(u_i\), we extend the singularity analysis to the generating function of \(u_i\).

Let \((u_i(k))_{k \in \mathbb N}\) denote the limiting eigenvectors for \(i \in [N]\), i.e., a positive fixed point of \(\mathcal T_{\lambda_\infty}\), whose existence is guaranteed by Proposition \ref{prop:0}. We assume that \(u_1(1)=1\). For most of this section, we will not rely on any special property of the constant \(\lambda_{\infty}\); thus, one may instead consider an arbitrary \(\lambda>0\) and assume that \((u_i(k))_{k \in \mathbb N}\) is a positive fixed point of \(\mathcal T_{\lambda}\). The normalization \(u_1(1)=1\) is imposed to ensure that each component of the fixed point is bounded away from zero.

We define the generating function on \(\mathbb C\) corresponding to the sequence \((u_i(k))_{k \in \mathbb N}\) by
\[
U_i(z) := \sum_{k \in \mathbb N} u_i(k)z^k.
\]
Let $s(k)=\sum_{i \in [N]}u_i(k)$.
Similarly, we define the generating function corresponding to the bulk eigenvector term \(s(k)\), associated with \((u_i(k))_{k \in \mathbb N,\, i \in [N]}\), as
\begin{equation}\label{eq:S}
S(z) := \sum_{k\in \mathbb N} s(k)z^k = \sum_{i \in [N]} U_i(z).
\end{equation}
With these generating functions, our aim is to express \(U_i(z)\) in terms of \(S(z)\). We will study the behavior of the coefficients \(u_i(k)\) via the series representation of \(U_i\), for which we require a suitable radius of convergence.

\begin{lemma}\label{lem:14}
For each \(i \in [N]\), let \(R_i\) denote the radius of convergence of \(U_i\) about \(z=0\). Then \(R_i > \beta_{\max}\) for all \(i \in [N]\). Moreover, \(R_i = R_j\) for all \(i,j \in [N]\).
\end{lemma}
\begin{proof}
   The proof closely follows that of Lemma \ref{lem:beta_weighted_sum_conv}. The key step is establishing the strict inequality \(R_i > \beta_{\max}\). We first show that for each \(i \in [N]\), the radius of convergence satisfies \(R_i > \beta_i\). Using the eigenvector equalities, note that for all \(k\in \mathbb N\), using the boundary conditions that correspond to the population $i$, we must have
\begin{equation}\label{sign}
u_i(k) \ge \frac{\overline L_iK_i}{\rho_i\lambda_\infty}\sum_{j\ge 1}\beta_i^j\,u_i(k+j).
\end{equation}
Let \(c_i := \frac{\rho_i\lambda_\infty}{\overline L_i K_i}\) and define
\[
S_k := \sum_{m\ge k} \beta_i^m\, u_i(m).
\]
    Note that by \eqref{sign} we have
\begin{equation}\label{eq:d1}
S_{k+1} = \sum_{m\ge k+1} \beta_i^m u_i(m) = \beta_i^k \sum_{m\ge 1} \beta_i^m u_i(k+m) \le c_i \beta_i^k u_i(k),
\end{equation}
and \(S_k = \beta_i^k u_i(k) + S_{k+1}\). This gives
\begin{equation}\label{eq:d2}
\beta_i^k u_i(k) = S_k- S_{k+1}.
\end{equation}
Using \eqref{eq:d1} and \eqref{eq:d2}, we obtain
\[
S_{k+1} \le c_i(S_k-S_{k+1})
\ \Longrightarrow\
(1+c_i) S_{k+1} \le c_iS_k
\ \Longrightarrow\
S_k \le S_2\left( \frac{c_i}{1+c_i}\right)^{k-2}.
\]
Furthermore, since the eigenvector equations imply \(\beta_i^k u_i(k) \le S_k\), we obtain
\[
u_i(k) \le \frac{S_2}{\beta_i}\left( \frac{c_i}{(1+c_i)\beta_i}\right)^{k-2}.
\]
Therefore, by the root test for the radius of convergence,
\[
\limsup_{k \to \infty} u_i(k)^{1/k} \le  \frac{c_i}{(1+c_i)\beta_i}
\ \Longrightarrow\
R_i \ge \frac{ (1+c_i)\beta_i}{c_i} > \beta_i.
\]

Since \(R_i > \beta_i\) for every \(i \in [N]\), we also have
\[
\lambda_{\infty} u_j(k) \ge \frac{\overline L_j K_j}{\rho_j}\beta_j\,u_j(k).
\]
This implies that the corresponding power series have the same radius of convergence. Hence,
$
R_i = R_j$ for all  $i,j \in [N].$
Therefore, $R_i > \beta_{\max}$
for all  $i \in [N],$
as desired.
\end{proof}

By Lemma \ref{lem:14}, the generating function \(S(z)\) also has radius of convergence \(R_S > \beta_{\max}\). From now on, we write \(R\) for the common radius of convergence of \(U_i\) and \(S\), since they share the same radius of convergence. We now aim to obtain a more explicit representation of \(S\) using the fixed-point equation for \(\mathcal T_{\lambda_\infty}\).

We recall that the fixed points of the limiting eigenvector equation satisfies the following for $k\ge2$:
\begin{equation}\label{eq:limit}
    \lambda_\infty u_i(k) = \underbrace{\left(K_i+\frac{\overline L_iK_i}{\rho_i}\right)S(k-1)}_{I}+\underbrace{(\overline K_i-K_i)u_i(k-1)}_{II} + \underbrace{\frac{\overline L_iK_i}{\rho_i}\sum_{m\ge 1} \beta^m_i S(k+m-1)}_{III}.
\end{equation}
The LHS of \eqref{eq:limit} can be used to obtain the following generating function
\begin{equation}
\lambda_\infty \sum_{k \ge 2} u_i(k) = \lambda_\infty (U_i(z) - u_i(1)z).
\end{equation}

Similarly, using the term \(I\), we obtain the generating function
\begin{equation}
\sum_{k\ge 2} \left(K_i+\frac{\overline L_iK_i}{\rho_i}\right)S(k-1)z^k=\left(K_i+\frac{\overline L_iK_i}{\rho_i}\right)zS(z),
\end{equation}
which trivially has the same radius of convergence \(R\) as \(S\) and \(U_i\).
Using the term \(II\), we obtain
\begin{equation}
\sum_{k \ge 2} (\overline K_i-K_i)u_i(k-1)z^k = (\overline K_i-K_i)zU_i(z),
\end{equation}
which also has the radius of convergence \(R\).
It remains to construct a generating function for the term \(III\).

First, define
\[
y_i(k) = \sum_{m\ge 1} \beta_i^m\, s(k+m-1).
\]
We use \(y_i(k)\) to construct the generating function
\begin{align}\label{eq:d4}
Y_i(z)=\sum_{k \ge 2} y_i(k)z^k 
&= \sum_{k \ge 2}\sum_{m \ge 1} \beta_i^m\, s(k+m-1)z^k .
\end{align}
Note that for each \(i\in [N]\), the radius of convergence of \(Y_i\) about \(z=0\) is \(R\). This can be deduced from \eqref{eq:limit}, since terms \(I\) and \(II\) induce generating functions with radius of convergence \(R\), as discussed above.

Thus, \(Y_i(z)\) is analytic in a neighborhood of the origin, specifically on the open disk of radius \(R>\beta_{\max}\) around $z=0$. Our aim is to show that \(S\) admits a meromorphic extension to the whole complex plane, which will imply that the poles of $S$ on the complex plane must have a finite order. To this end, we first provide a more explicit representation for \(Y_i\) on its domain of analyticity by arguing through algebraic manipulations and analytic extensions, which will then be used to recover an explicit representation for $S$ in its radius of convergence.
\begin{lemma}\label{lem:15}
    For all $|z| <R$, we have
    \[
    Y_i(z) = \frac{\beta_i z}{z-\beta_i}\left(S(z)-S(\beta_i) \right).
    \]
\end{lemma}
\begin{proof}
Note that on the annulus $R>|z|> \beta_{\max}$, applying the Fubini-Tonelli theorem to \eqref{eq:d4} (which is applicable due to Lemma \ref{lem:14}) we   
\begin{align}
Y_i(z) &= \sum_{m \ge 1}\beta^m_i\sum_{ k \ge 2}s(k+m-1)z^k \\&= \sum_{m \ge 1}\beta^m_i\sum_{j\ge m+1}s(j)z^{j+1-m}
\\&=\sum_{m \ge 1} \frac{\beta_i}{z^m}z\left(S(z)-\sum_{n=1}^{m-1}s(n)z^n\right)
\\&=\frac{\beta_i}{1-\frac{\beta_i}z}S(z)-\sum_{m \ge 1}\beta^m_i z^{1-m}\sum_{n=1}^{m-1}s(n)z^n\label{eq:40}.
\end{align}

Define 
\[
\tilde Y_i(z) = \frac{\beta_i }{1-\frac{\beta_i}z}S(z)-\sum_{m \ge 1}\beta^m_i z^{1-m}\sum_{n=1}^{m-1}s(n)z^n,
\]
which is analytic on the annulus $R>|z|> \beta_i$ and on the disc $|z|< \beta_i$. As shown above, by the identity theorem we have $\tilde Y_i(z) = Y_i(z)$ on $R>|z|> \max_j \beta_j$. Since $Y_i(z)$ is analytic on $|z|<R$, it follows that $|z|=\beta$ is not a singularity region for $\tilde Y_i(z)$. Indeed, on the disc $|z|<R$, $\tilde Y_i(z)$ admits the analytic extension
\[
\hat Y_i(z)=\left(\frac{\beta_i }{1-\frac{\beta_i}z}S(z)-\sum_{m \ge 1}\beta^m_i z^{1-m}\sum_{n=1}^{m-1}s(n)z^n\right)1_{\{R>|z|>\beta_i\}}+Y_i(z)1_{\{|z|\le\beta_i\}}.
\]
Since $\hat Y_i(z)$ is analytic on  $|z|<\beta_i$, it follows that $\hat Y_i(z)$ is an analytic extension of $\tilde Y_i(z)$ to the disc $|z|< R$. Indeed, by definition, we have $\hat Y_i(z) = \tilde Y_i(z)$ on the annulus $\max_j\beta_j<|z|<R$. Thus, it follows that $\hat Y_i(z) = \tilde Y_i(z)$ on $|z|<R$ by the identity theorem. To show this, note that, using the fact that $S(z)$ is analytic on the annulus $0<|z-\beta_i|<\varepsilon,$ we obtain that $z=\beta_i$ is a removable singularity for $\tilde Y_i(z)$. Then, since $\tilde Y_i(z)$ is also analytic on $|z|<\max_j \beta_j$, it follows that $\tilde Y_i(z) = Y_i(z)$ on $|z| <R$ for all $i \in [N]$ by the identity theorem (after canceling the removable singularities). Thus, $\hat Y_i(z) = \tilde Y_i(z)$ on $|z|<R$.

It remains to control the term $\sum_{m \ge 1}\beta^m_i z^{1-m}\sum_{n=1}^{m-1}S(n)z^n$. Note that when $|z| > \beta_i$, we have
\begin{align*}
\sum_{m \ge 1}\beta^m_i z^{1-m}\sum_{n=1}^{m-1}s(n)z^n &= \sum_{m \ge 2} \sum_{n=1}^{m-1}\beta^m_is(n)z^{1-m+n} = \sum_{n \ge 1} s(n) \sum_{m \ge n+1} \beta^m_i z^{n+1-m}
\\&=\sum_{n\ge 1}s(n)\sum_{t \ge 0} \beta^{t+n+1}z^{-t}=\sum_{n \ge 1}s(n)\beta^{n+1}_i \sum_{t \ge 0}\left( \frac{\beta_i}{z} \right)^t
\\& = \sum_{n \ge 1}s(n)\beta^{n+1}_i \frac{1}{1-\beta_i/z} = \frac{\beta_i}{1-\beta_i/z}S(\beta_i).
\end{align*}
Using the same analytic continuation argument above, we then obtain 
\[
Y_i(z) = \frac{\beta_i z(S(z)-S(\beta_i))}{z-\beta_i}
\]
on the whole $|z| < R$, as desired.
\end{proof}

Using generating functions to rewrite \eqref{eq:limit}, we obtain
\begin{equation}\label{eq:limit4}
\left( \lambda_{\infty}-(\overline K_i-K_i)z\right) U_i(z) = \lambda_{\infty}u_i(1)z + \left(K_i+\frac{\overline L_iK_i}{\rho_i}\right)zS(z)+\frac{\overline L_iK_i}{\rho_i}Y_i(z).
\end{equation}
Let 
\begin{equation}\label{eq:bb}
H_i(z) := \frac{\lambda_{\infty}u_i(1)z+\frac{\overline L_iK_i}{\rho_i}\frac{\beta_i z}{z-\beta_i}(S(z)-S(\beta_i))}{\lambda_{\infty}-(\overline K_i-K_i)z}, \quad , B_i(z):= \frac{( K_i+\frac{\overline L_iK_i}{\rho_i})z}{\lambda_{\infty}-(\overline K_i-K_i)z}.
\end{equation}
Note that $H_i(z)$ is analytic over $| z|<R$ by Lemma \ref{lem:15}. Next, we prove that $S$ admits a meromorphic extension, which in turn implies that the only singularities of $S$ are poles. Furthermore, we will identify an upper bound for the radius of convergence $R$. We will need these explicit bounds on $R$ to relate $\lambda_{\infty}$ with some $\rho(B(r))$. Before proceeding, for technical purposes, we need the inequality $\min_i \frac{\lambda_{\infty}}{\overline K_i-K_i} > \beta_{\max}$. 

\begin{lemma}\label{lem:60}
By \((u_i(k))_{i,j}\), denote a positive fixed point of \(T_{\lambda_\infty}\).
Then, it holds that
\[
\min_{i\in[N]}\frac{\lambda_\infty}{\overline K_i-K_i}>\beta_{\max}.
\]
\end{lemma}

\begin{proof}
Let \(p \in [N]\) be such that \(\beta_{\max} = \beta_p\). Since \(u\) is a positive fixed point of \(T_{\lambda_\infty}\), \eqref{eq:BC} gives
\[
u_p(1)
=
\frac{\overline L_pK_p}{\rho_p\lambda_\infty}
\sum_{k\ge 1}\beta_p^k s(k)
=
\frac{\overline L_pK_p}{\rho_p\lambda_\infty}
\sum_{k\ge 1}\beta_{\max}^k s(k).
\]
Because \(u_p(1)<\infty\), we obtain $\sum_{k\ge 1}\beta_{\max}^k s(k)<\infty.$

Now fix \(i\in[N]\). For \(k\ge 2\), \eqref{eq:BULK} gives
\[
u_i(k)
=
\frac{\overline K_i-K_i}{\lambda_\infty}u_i(k-1)
+
\frac{K_i}{\lambda_\infty}s(k-1)
+
\frac{\overline L_iK_i}{\rho_i\lambda_\infty}
\sum_{m\ge 0}\beta_i^m s(k+m-1).
\]
Since all terms on the right-hand side are nonnegative,
\[
u_i(k)
\ge
\frac{\overline K_i-K_i}{\lambda_\infty}u_i(k-1) \implies 
u_i(k)
\ge
u_i(1)
\left(
\frac{\overline K_i-K_i}{\lambda_\infty}
\right)^{k-1}
\qquad \text{for all } k\ge 1.
\]

Suppose, for the sake of contradiction, that $\frac{\lambda_\infty}{\overline K_i-K_i}\le \beta_{\max}$
for some \(i\in[N]\).
Therefore,
\[
\sum_{k\ge 1}\beta_{\max}^k u_i(k)
\ge
u_i(1)\beta_{\max}
\sum_{k\ge 1}
\left(
\beta_{\max}
\frac{\overline K_i-K_i}{\lambda_\infty}
\right)^{k-1}
=
\infty.
\]
Since $s(k)\ge u_i(k),$ we get $\sum_{k\ge 1}\beta_{\max}^k s(k)=\infty,$ which contradicts
$\sum_{k\ge 1}\beta_{\max}^k s(k)<\infty,$ see Lemma \ref{lem:14}.
Hence, for every \(i\in[N]\), we must have $\frac{\lambda_\infty}{\overline K_i-K_i}>\beta_{\max}.$
The desired conclusion follows after taking the minimum over \(i\in[N]\).
\end{proof}

\begin{lemma}\label{lem:16}  
    $S$ admits a meromorphic extension and on $|z|<R$ it satisfies
    \begin{equation}\label{eq:s(z)}
    S(z) = \frac{\sum_{i \in [N]}\frac{ \lambda_{\infty}u_i(1)z-\frac{\overline L_iK_i\beta_i z}{\rho_i(z-\beta_i)}S(\beta_i)}{\lambda_{\infty}-(\overline K_i-K_i)z}}{1 - \sum_{i \in [N]}\frac{(K_i+\frac{\overline L_iK_i}{\rho_i})z+\frac{\overline L_iK_i\beta_iz}{\rho_i(z-\beta_i)}}{\lambda_{\infty}-(\overline K_i-K_i)z}}.
    \end{equation}
    In particular, singularities of $S$ are either removable or a pole of finite order. Furthermore,
    $R\le\min_i  \frac{\lambda_{\infty}}{\overline K_i-K_i}.$
\end{lemma}
\begin{proof}
The equation \eqref{eq:limit4} implies that
\begin{equation}\label{eq:limit2}
U_i(z) = H_i(z) + B_i(z)S(z).
\end{equation}
Using the definition of $S$ and \eqref{eq:limit2}, we can express $S$ as
\begin{equation}\label{eq:ref}
S(z) = \sum_{i \in [N]} U_i(z) = \sum_{i \in [N]}\left(\frac{ \lambda_{\infty}u_i(1)z-\frac{\overline L_iK_i}{\rho_i}\frac{\beta_i z}{z-\beta_i}S(\beta_i)}{\lambda_{\infty}-(\overline K_i-K_i)z} + S(z)\frac{( K_i+\frac{\overline L_iK_i}{\rho_i})z+\frac{\overline L_iK_i}{\rho_i}\frac{\beta_iz}{z-\beta_i}}{\lambda_{\infty}-(\overline K_i-K_i)z}\right).
\end{equation}
It follows that \eqref{eq:s(z)} holds on $|z|<R.$
Thus, $S$ admits a meromorphic extension by the identity theorem.

It remains to show that $R \le \min_i\frac{\lambda_{\infty}}{\overline K_i-K_i}$. Note that $\frac{\overline L_iK_i}{\rho_i}S(\beta_i) = \lambda_{\infty} u_i(1)$ as a consequence of the boundary conditions of $\mathcal T_{\lambda_{\infty}}$, see \eqref{eq:boundary}. Thus, 
\[
\lambda_\infty u_i(1)z -\frac{\overline L_iK_i}{\rho_i}\frac{\beta_i z}{z-\beta_i}S(\beta_i) = \frac{\lambda_{\infty}u_i(1)z^2}{z-\beta_i}.
\]
Since $S(x)>0$ for any $x>0$, note that for any $x \in (\beta_{\max}, \min_i\frac{\lambda_{\infty}}{ \overline K_i-K_i})$, \eqref{eq:ref} implies that
\[
S(x) \ge \sum_{i \in [N]} \frac{\lambda_{\infty}u_i(1)x^2}{(x-\beta_i)(\lambda_{\infty}-(\overline K_i-K_i)x)}.
\]
By Lemma \ref{lem:60}, it holds that $\lambda_{\infty}/(\overline K_i- K_i) >1$ for all $i \in [N]$; thus, the limit 
\begin{equation}\label{eq:crit}
\lim_{x \to \left(\min_i \frac{\lambda_{\infty}}{\overline K_i-K_i}\right)^-} \sum_{i \in [N]} \frac{\lambda_{\infty}u_i(1)x^2}{(x-\beta_i)(\lambda_{\infty}-(\overline K_i-K_i)x)} = \infty
\end{equation}
holds;
 therefore, we must have $R \le \min_i \frac{\lambda_{\infty}}{\overline K_i-K_i},$ where the limit in \eqref{eq:crit} is taken over the real line.
\end{proof}

Next, using the characterization of the poles of $S$ and the existence of the eigenvector $(u_i(k))_{i,k}$, we prove that the ``smallest'' pole of $S$ must satisfy a relation with $\lambda_{\infty}$ similar to the relation presented between $\rho(B(r))$ and $r$ in Lemma \ref{lem:beta_weighted_sum_conv}.

\begin{proposition}\label{prop:1}
    There exists $z_0 \in (\beta_{\max},\min_i \frac{\lambda_{\infty}}{\overline K_i-K_i} )$ such that
    \[
    1 = \sum_{i \in [N]}\frac{( K_i+\frac{\overline L_iK_i}{\rho_i})z_0+\frac{\overline L_iK_i}{\rho_i}\frac{\beta_iz_0}{z_0-\beta_i}}{\lambda_{\infty}-(\overline K_i-K_i)z_0}.
    \]
    Furthermore, $z_0$ is {an isolated pole} of finite order of $S$. 
\end{proposition}
\begin{proof}
By Lemma \ref{lem:14}, the singularities $z=\beta_i$ of $S$, $i \in [N]$, are removable. Thus, all non-removable singularities of $S$ come from $z$ such that $1 = \sum_{i \in [N]}\frac{( K_i+\frac{\overline L_iK_i}{\rho_i})z+\frac{\overline L_iK_i}{\rho_i}\frac{ \beta_iz}{z-\beta_i}}{\lambda_{\infty}-(\overline K_i-K_i)z}.$ We recall that for a power series of the form $f(z) = \sum_n a_n z^n,$ Vivanti-Pringsheim theorem says that if the radius of convergence of $f$ is $R<\infty$, and $a_n \ge0$ for all $n$, then $z=R$ is a singularity of $f$.
Since $(u_i(k))_{i,k}$ are all positive, using the definition \eqref{eq:S} and Vivanti-Pringsheim theorem, we see that $S$ must have a singularity at the boundary $z=R$, which corresponds to a value that satisfies $1 = \sum_{i \in [N]}\frac{( K_i+\frac{\overline L_iK_i}{\rho_i})z+\frac{\overline L_iK_i}{\rho_i}\frac{ \beta_iz}{z-\beta_i}}{\lambda_{\infty}-(\overline K_i-K_i)z}$ when $R<\infty$.

First, we will show that in a disc with radius $R< \infty$, there is a non-removable singularity of $S$. Indeed, the equation \eqref{eq:crit} in Lemma \ref{lem:16} shows that $z= \min_i \frac{\lambda_\infty}{\overline K_i-K_i}$ is a non-removable singularity of $S$, by definition. Thus, we must have $R< \infty$, meaning that there exists $z_0 \in (0,\infty)$ such that
\begin{equation}\label{eq:fff}
1 = \sum_{i \in [N]}\frac{( K_i+\frac{\overline L_iK_i}{\rho_i})z_0+\frac{\overline L_iK_i}{\rho_i}\frac{ \beta_iz_0}{z_0-\beta_i}}{\lambda_{\infty}-(\overline K_i-K_i)z_0} = F(z_0)
\end{equation}
that is not canceled by the numerator of $S(z)$.
Note that \eqref{eq:fff} must hold as by Lemma \ref{lem:14}, we also must have $z_0 > \beta_{\max}$, meaning that the numerator of $S$ does not have any singularity. The upper bound on $z_0$ follows from Lemma \ref{lem:16}. The upper bound of $z_0$ is not sharp since \eqref{eq:fff} is not feasible at the boundary $z= \min_i\frac{\lambda_{\infty}}{\overline K_i-K_i}$.

Since the denominator of $S$ is a nonzero rational function, each of its zeros has finite multiplicity. Therefore, every non-removable singularity of $S$ arising from a zero of its denominator is a pole of finite order.

{It remains to show that $z_0$ is an isolated pole. To prove this, we proceed by contradiction and suppose otherwise. If $|w|=R$ is also a singularity of $S$, then $F(w)=1$ by \eqref{eq:fff}. By the reverse triangle inequality, we must have $|w-\beta_i| > \left | |w| - \beta_i \right| = R-\beta _i$. Similarly, we also have $|\lambda_{\infty}-(\overline K_i-K_i)w|\ge \lambda_{\infty}-(\overline K_i-K_i)R$. With these observations, note that
\begin{align*}
1=\left|
\sum_{i \in [N]}\frac{
K_iw+\dfrac{\overline L_i K_iw^2}{\rho_i(w-\beta_i)}
}{
\lambda_\infty-(\overline K_i-K_i)w
}
\right|
&\le
\sum_{i \in [N]}\frac{
K_iR+\dfrac{\overline L_i K_iR^2}{\rho_i|w-\beta_i|}
}{
|\lambda_\infty-(\overline K_i-K_i)w|
}
< \sum_{i \in [N]}\frac{K_iR+\dfrac{\bar L_iK_iR^2}{\rho_i(R-\beta_i)}}{\lambda_\infty-(\overline K_i-K_i)R}=F(R)=1,
\end{align*}
which is a contradiction. Thus, since $S$ is meromorphic, it follows that $z_0=R$ is an isolated pole.}
\end{proof}

Proposition \ref{prop:1} shows that if $\mathcal T_{\lambda_\infty}$ admits a positive fixed-point $u$, then 
\[
1 = \sum_{i \in [N]}\frac{K_iz_0+\frac{\overline L_iK_i}{\rho_i}\frac{z_0^2}{z_0-\beta_i}}{\lambda_{\infty}-(\overline K_i-K_i)z_0} 
\]
must hold for $\lambda_{\infty}$ for some $z_0$.
Next, we will prove the converse, i.e.  if $\lambda_{\infty}$ satisfies
\[
    1 = \sum_{i \in [N]}\frac{ K_iz_0+\frac{\overline L_iK_i}{\rho_i}\frac{z_0^2}{z_0-\beta_i}}{\lambda_{\infty}-(\overline K_i-K_i)z_0}
    \]
for some $z_0 \in (\beta_{\max},\infty),$ then there exists a positive fixed point for $\mathcal T_{\lambda_\infty}$. For this purpose, we will rely on standard fixed-point arguments and define a compact domain for $\mathcal T_{\lambda_\infty}$.

This procedure will rely on an intermediate step to relate $\lambda_{\infty}$ with $\rho(B(r))$ for some $r$. To be more precise, our aim will be to prove that
\begin{equation}\label{eq:const}
\lim_{k\to\infty}\frac{u_i(k+1)}{u_i(k)}=\frac{1}{z_0}\qquad\text{for each }i \in [N],
\end{equation}
where $z_0$ is the same for all $i \in [N].$
The proof will solely rely on singularity analysis, so first, we recall some basic results from complex analysis that will be essential in proving \eqref{eq:const}.

\begin{lemma}[Cauchy's coefficient formula]\label{lem:50}
Let $f$ be analytic on $\{z:|z|\le r\}$. By $[z^k]f(z)=a^k$ denote the coefficient of the term $z^k$ on the power series representation of $f= \sum_n a_nz^n$. Then for each $k\ge 0$,
\[
[z^k]f(z)=\frac{1}{2\pi i}\int_{|z|=r}\frac{f(z)}{z^{k+1}}\,dz.
\]
\end{lemma}
\begin{proof}
    See \cite[Chapter 4, Section 2, Lemma 3]{ahlfors1979complex}.
\end{proof}
A well known consequence of Lemma \ref{lem:50} is the following boundary estimate, which is also due to Cauchy:
\begin{lemma}\label{lem:51}
If $f$ is analytic on $\{|z|\le r\}$ and $M_r:=\max_{|z|=r}|f(z)|$, then
\[
|[z^k]f(z)|\le \frac{M_r}{r^k}\qquad(k\ge 0).
\]
\end{lemma}
\begin{proof} Apply Lemma \ref{lem:50} and bound:
\[
\left|[z^k]f(z)\right|
\le \frac{1}{2\pi}\int_{|z|=r}\frac{|f(z)|}{|z|^{k+1}}\,|dz|
\le \frac{1}{2\pi}\int_{|z|=r}\frac{M_r}{r^{k+1}}\,|dz|
=\frac{M_r}{r^k}.
\]
\end{proof}

In Lemma \ref{lem:16}, we showed that $S$ only admits poles as non-removable singularities. To obtain \eqref{eq:const}, we will first obtain this result for the bulk eigenvector term by using basic singularity analysis by using the tools mentioned above and the residue theorem.

{
\begin{lemma}\label{lem:52}
Let the order of the pole $z_0$ on $S$ be $p$. Then, there exist a polynomial $P$ of order $p-1$ such that
\[
s(k)=z_0^{-k}P(k)+O((R+\varepsilon)^{-k}).
\]
Consequently,
\[
\frac{s(k+1)}{s(k)}\to \frac{1}{z_0}.
\]
\end{lemma}
\begin{proof}
By Proposition \ref{prop:1}, $S$ has a pole of finite order at $z_0=R$ isolated from others. Thus, by Lemma \ref{lem:16}, there exist $(c_{\ell})_{\ell \in [p]}\neq 0$ and an analytic function $G$ on $|z| < R+\varepsilon$ such that
\begin{equation}\label{eq:S-decomp}
S(z)=\sum_{\ell=1}^{p} \frac{c_\ell}{(1-z/z_0)^\ell}+G(z).
\end{equation}
Fix the complex circle $|z|=(R+\varepsilon)$ (with $R+\varepsilon>z_0$). By Proposition \ref{prop:1}, we can choose $\varepsilon>0$ so that the only singularity of $S$ inside the disc $|z| \le R+\varepsilon$ is $z_0$. Then, using Lemma \ref{lem:50}, we obtain
\begin{align*}
    s(k) = [z^k]S(z) 
    &= \sum_{\ell=1}^{p} c_\ell [z^k](1-z/z_0)^{-\ell} + [z^k]G(z)
    \\& =z_0^{-k} \sum_{\ell=1}^{p} c_\ell \binom{k+\ell-1}{\ell-1} + [z^k]G(z),
\end{align*}
where we have used \[
[z^k](1-z/z_0)^{-\ell} = \binom{k+\ell-1}{\ell-1}z_0^{-k} \]
on the last line.
Since \(G\) is analytic on \(|z|<R+\varepsilon\), Lemma~\ref{lem:51} implies that 
\[
\left|\frac{1}{2\pi i}\int_{|z|=R+\varepsilon}\frac{G(z)}{z^{k+1}}\,dz\right|
=|[z^k]G(z)|\le \frac{M_{R+\varepsilon}(G)}{(R+\varepsilon)^k}.
\]
Consequently, 
\begin{equation}\label{eq:suyo}
s(k) = z_0^{-k} \underbrace{\sum_{\ell=1}^{p} c_\ell \binom{k+\ell-1}{\ell-1}}_{=P(k)} + O\bigl((R+\varepsilon)^{-k}\bigr).
\end{equation}

Since \(c_p\neq0\), the polynomial \(P\) has degree \(p-1\), with \[ P(k) = \frac{c_p}{(p-1)!}k^{p-1} + O(k^{p-2}). \] Hence, $\frac{P(k+1)}{P(k)} \rightarrow 1. $ Moreover, since \(z_0<R+\varepsilon\), \[ \frac{O((R+\varepsilon)^{-k})} {z_0^{-k}P(k)} = O\left( \frac{(z_0/(R+\varepsilon))^k}{k^{p-1}} \right) \longrightarrow 0. \] Therefore, \begin{align*} 
\frac{s(k+1)}{s(k)} &= \frac{ z_0^{-(k+1)}P(k+1) + O((R+\varepsilon)^{-(k+1)}) }{ z_0^{-k}P(k) + O((R+\varepsilon)^{-k}) } \\ &= \frac{1}{z_0} \frac{P(k+1)}{P(k)} \frac{1+o(1)}{1+o(1)} \longrightarrow \frac{1}{z_0},\end{align*}
as desired.
\end{proof}
}

Next, using equation \eqref{eq:limit4}, we will use a similar singularity analysis on $U_i(z)$ through $S(z)$ to obtain \eqref{eq:const} since Lemma \ref{lem:16} implies that the singularities of $U_i(z)$ must come from $S(z)$. We will also identify an asymptotic cross-interaction between the terms $u_i(k)$ and $u_j(k)$ between different populations that will be useful while identifying $z_0$ such that $\lambda_\infty = \rho(B(1/z_0))$.

{
\begin{lemma}\label{lem:54}
Let $p\geq 1$ denote the order of the pole of $S$ at $z_0=R$. For every $i\in[N]$, there exists a polynomial $P_i$ of degree $p-1$, with positive leading coefficient $\gamma_i$, such that 
\[
u_i(k) = z_0^{-k}P_i(k) + O\bigl((R+\varepsilon)^{-k}\bigr). 
\]
Consequently, for every $i\in [N]$
\[
\lim_{k\to\infty} \frac{u_i(k+1)}{u_i(k)} = \frac{1}{z_0}.
\] 
Furthermore, for all $i,j \in [N]$ we obtain 
\[
\lim_{k\to\infty}\frac{u_i(k)}{u_j(k)} = \frac{\gamma_i}{\gamma_j}.
\]
\end{lemma}

\begin{proof}
Using \eqref{eq:limit} and \eqref{eq:bb}, we have 
\[
U_i(z) = \frac{\lambda_{\infty}u_i(1)z^2}{(\lambda_{\infty}-(\overline K_i-K_i)z)(z-\beta_i)} + S(z) \underbrace{\left( \frac{( K_i+\frac{\overline L_iK_i}{\rho_i})z+\frac{\overline L_iK_i}{\rho_i}\frac{\beta_iz}{z-\beta_i}}{\lambda_{\infty}-(\overline K_i-K_i)z} \right)}_{=A_i}.
\]
Since $z_0\in \left( \beta_{\max}, \min_i \frac{\lambda_\infty}{\overline K_i-K_i} \right),$ the function $A_i$ is analytic at $z_0$, and $A_i(z_0)>0.$ By Proposition~\ref{prop:1}, $S$ has a pole of order $p$ at $z_0$. Thus, for some constants \(c_1,\ldots,c_p\), with \(c_p\neq0\), \[ S(z) = \sum_{\ell=1}^{p} \frac{c_\ell}{(1-z/z_0)^\ell} + G(z), \] where \(G\) is analytic in a neighborhood extending beyond \(|z|=z_0\). Since \(A_i\) is analytic and nonzero at \(z_0\), multiplication by \(A_i\) preserves the order of the pole. Hence by \eqref{eq:S-decomp}, for $c_{i,p}=A_i(z_0)c_p$, we obtain 
\[
U_i(z) = \sum_{\ell=1}^{p} \frac{c_{i,\ell}}{(1-z/z_0)^\ell} + G_i(z). 
\] 
Therefore, 
\begin{align*}
u_i(k) = [z^k]U_i(z) = z_0^{-k} \underbrace{\sum_{\ell=1}^{p} c_{i,\ell} \binom{k+\ell-1}{\ell-1}}_{=P_i(k)} + [z^k]G_i(z). 
\end{align*}
Similar to \eqref{eq:suyo}, we then obtain 
\[
u_i(k) = z_0^{-k}P_i(k) + O\bigl((R+\varepsilon)^{-k}\bigr).
\]
Its leading coefficient is $\gamma_i = \frac{c_{i,p}}{(p-1)!}.$ As in the proof of Lemma \ref{lem:52}, we have $\frac{P_i(k+1)}{P_i(k)}\longrightarrow 1,$ and, since $z_0<R+\varepsilon$, $\frac{O((R+\varepsilon)^{-k})} {z_0^{-k}P_i(k)} \rightarrow0.$ Consequently, $\frac{u_i(k+1)}{u_i(k)} \rightarrow \frac1{z_0}.$ Finally, since $P_i(k) = \gamma_i k^{p-1}(1+o(1))$ for every \(i\in[N]\), we obtain \[ \frac{u_i(k)}{u_j(k)} = \frac{P_i(k)+o(k^{p-1})} {P_j(k)+o(k^{p-1})} \longrightarrow \frac{\gamma_i}{\gamma_j}, \] which proves the claim. 
\end{proof}
}

By Lemma \ref{lem:54}, for every $i$,
\[
\lim_{k\to\infty}\frac{u_i(k+1)}{u_i(k)}=\frac{1}{z_0}.
\]
Define $r:=1/z_0$. Then $u_i(k+1)/u_i(k)\to r$ for all $i$, i.e. the ratio tends to
a constant independent of $i$. The same also applies to the asymptotic cross-interaction betweeen different population groups. Next, we will show that this implies that $\lambda_\infty = \rho(B(r))$. This allow us to relate Proposition \ref{prop:1} to Lemma \ref{lem:57}.

\begin{lemma}\label{lem:55}
    We have $\lambda_{\infty} = \rho(B(1/z_0)).$
\end{lemma}
\begin{proof}
    For all $k \ge 2$, using the positive fixed point of limiting eigenvector equation $\mathcal  T_{\lambda_{\infty}}$, which exists as shown in Proposition \ref{prop:0}, we obtain
    \[
    \lambda_{\infty} = (\overline K_i-K_i)\frac{u_i(k)}{u_i(k+1)} +\left(K_i+\frac{\overline L_iK_i}{\rho_i}\right)\sum_{j \in [N]} \frac{u_j(k)}{u_i(k+1)} + \frac{\overline L_iK_i}{\rho_i} \sum_{j \in [N]}\sum_{m \ge 1}\beta^m_i\frac{u_j(k+m)}{u_i(k+1)}.
    \]
    Using Lemma \ref{lem:54}, and the Fatou lemma, we then obtain that
    \[
    \lambda_{\infty} \begin{bmatrix}
        \gamma_1&
        \gamma_2&
        \gamma_3&
        \cdots&
        \gamma_N
    \end{bmatrix}^{\top} \ge  \left(\begin{bmatrix}
        \gamma_1&
        \gamma_2&
        \gamma_3&
        \cdots&
        \gamma_N
    \end{bmatrix}B(1/z_0)^{\top}\right)^{\top},
    \]
    where the constants $\gamma_i$ are the ones specified in Lemma \ref{lem:54}.
    Since $z_0>0,$ $\gamma_i>0$ and $\lambda_{\infty}>0$, as a consequence of the Perron-Frobenius theorem, we obtain that $\rho(B(1/z_0))\le\lambda_{\infty} \le \inf_{ z \ge \beta_{\max}} \rho(B(1/z))$, and thus it holds that $\rho(B(1/z_0))=\lambda_{\infty}$, as desired.
\end{proof}

\begin{proof}[Proof of Proposition \ref{prop:quant}]
    The result follows from Lemmas \ref{lem:8} and \ref{lem:55}.
\end{proof}

\begin{proof}[Proof of Theorem \ref{thrm:5}]
Throughout this subsection, we did not use any property specific to the constant \(\lambda_{\infty}\). Therefore, by Lemma~\ref{lem:55}, the same conclusion remains valid whenever \(\mathcal T_{\lambda}\) admits a positive fixed point for some \(\lambda>0\).
\end{proof}

Lemma \ref{lem:55} technically presents a complete characterization when $\lambda_{\infty}<1$, so we have completed the proof of one side of Theorem \ref{thrm:4} as expressed in Proposition \ref{prop:quant}. However, Proposition \ref{prop:quant} still requires calculating $\inf_{r}\rho(B(r)),$ which can be inconvenient, especially when $N$ is large as the matrix $B(r)$ is not sparse. Indeed, standard methods in the literature to calculate $\rho(B(r))$ requires $O(N^3)$ operations. 

To complete the proof of the other half of Theorem \ref{thrm:4}, we will need the following auxiliary lemma, which will allow us to directly relate $\lambda$ values for which $\mathcal T_{\lambda}$ admits a fixed point and the values $\rho(B(r))$.

{
\begin{lemma}\label{lem:unique-e} 
Fix $t>\beta_{\max}$. For
$e\in\mathbb R \setminus \left\{(\overline K_1-K_1)t,\dots,(\overline K_N-K_N)t \right\},$ define
$$
G(t,e) := \sum_{i=1}^N \frac{\left(K_i+\frac{\overline L_iK_i}{\rho_i}\right)t+\frac{\frac{\overline L_iK_i}{\rho_i}\beta_i t}{t-\beta_i}}{e-(\overline K_i-K_i)t
}.
$$
Then, there exists a unique
$e^{\ast}\in(t\overline K_\infty,\infty)$
such that $G(t,e^{\ast})=1$.
\end{lemma}

\begin{proof}
If \(e>t\overline K_\infty\), then for every \(i\), $e-(\overline K_i-K_i)t \ge e-t\overline K_\infty >0,$
and hence each $G(t,e)$ is finite. Therefore, $e\mapsto G(t,e)$ is well defined and continuous on
$(t\overline K_\infty,\infty)$.

For $e>t\overline K_\infty$, we may differentiate term-by-term to obtain
\[
\frac{\partial G}{\partial e}(t,e)
=
\sum_{i=1}^N
-
\frac{
\left(K_i+\frac{\overline L_iK_i}{\rho_i}\right)t
+
\frac{\frac{\overline L_iK_i}{\rho_i}\beta_i t}{t-\beta_i}
}{
\left(e-(\overline K_i-K_i)t\right)^2
}.
\]
Since
\[
\left(K_i+\frac{\overline L_iK_i}{\rho_i}\right)t
+
\frac{\frac{\overline L_iK_i}{\rho_i}\beta_i t}{t-\beta_i}
>0
\text{ and }
\left(e-(\overline K_i-K_i)t\right)^2>0
\]
for all $e>t\overline K_\infty$, for all $e>t \max_j(\overline K_j-K_j)$ we have $\frac{\partial G}{\partial e}(t,e)<0.$ Hence $G(t,\cdot)$ is strictly decreasing on
$(t\overline K_\infty,\infty)$.

Let $I:=\left\{i:\,\overline K_i-K_i=\overline K_\infty \right\}$, which is nonempty by definition of $\overline K_\infty$. As $e\to(t\overline K_\infty)^+$, for each $i\in I$ we have $e-(\overline K_i-K_i)t=e-t\overline K_\infty\to0^+,$ and thus $G(t,e)\to+\infty.$
Since all terms are nonnegative for $e>t\overline K_\infty$, it follows that
$\lim_{e\to(t\overline K_\infty)^+}G(t,e)=+\infty.$
On the other hand, as $e\to\infty$, we have for each fixed $i$, $G(t,e)\to0,$
and therefore $\lim_{e\to\infty}G(t,e)=0.$

We have thus shown that the function $G(t,\cdot)$ is continuous on
$(t\overline K_\infty,\infty)$ and satisfies
\[
\lim_{e\to(t\overline K_\infty)^+}G(t,e)=+\infty
\qquad\text{and}\qquad
\lim_{e\to\infty}G(t,e)=0.
\]
Hence, by the intermediate value theorem, there exists
$e^{\ast}\in(t\overline K_\infty,\infty)$
such that $G(t,e^{\ast})=1$.

As we have shown that $G(t,\cdot)$ is strictly decreasing on
$(t\overline K_\infty,\infty)$, it can take the value $1$ at most once on this interval. Therefore the solution $e^{\ast}$ is unique in this case.
\end{proof}
}

Next, we will prove a result that will amount to the proof of Theorem \ref{thrm:4}.

\begin{lemma}\label{lem:61}
    We have $\inf_{ 0<r < \beta_{\max}^{-1}}\rho(B(r))<1$ if and only if 
    \begin{equation}\label{eq:final}
    \inf_{ \max_j (\overline K_j-K_j) < r < \beta_{\max}^{-1}} \sum_{i \in [N]}\frac{\frac{ K_i+\frac{\overline L_iK_i}{\rho_i}}{r}+\frac{\overline L_iK_i\beta_i}{\rho_i}\frac{1}{1-\beta_i r}}{1-\frac{\overline K_i-K_i}{r}} < 1.
    \end{equation}
    Nonemptiness of the interval $(\max_j (\overline K_j-K_j), \beta_{\max}^{-1})$ is automatic when $\inf_{ 0<r < \beta_{\max}^{-1}}\rho(B(r))<1$, i.e., it is part of the contraction criterion.
\end{lemma}
\begin{proof} Suppose first that 
$\inf_{0<r<\beta_{\max}^{-1}}\rho(B(r))<1.$ 
Then there exists some $r\in(0,\beta_{\max}^{-1})$ such that $\rho(B(r))<1.$ By Lemma~\ref{lem:57}, \[ 1= \sum_{i\in[N]} \frac{ \frac{K_i+\frac{\overline L_i K_i}{\rho_i}}{r} + \frac{\overline L_iK_i\beta_i}{\rho_i(1-\beta_i r)} }{ \rho(B(r))-\frac{\overline K_i-K_i}{r} }. \] Moreover, Lemma~\ref{lem:57} implies that $r> \max_{j\in[N]} \frac{\overline K_j-K_j}{\rho(B(r))}. $ Since $\rho(B(r))<1$, it follows that $r>\max_{j\in[N]}(\overline K_j-K_j).$ Thus, $r$ belongs to the domain in \eqref{eq:final}. Furthermore, since $\rho(B(r))<1$, 
\[
\rho(B(r))-\frac{\overline K_i-K_i}{r} < 1-\frac{\overline K_i-K_i}{r} \]
for every $i\in[N]$. All the denominators above are positive by Lemma~\ref{lem:57}; hence, 
\[
1 > \inf_{\max_j(\overline K_j-K_j)<r<\beta_{\max}^{-1}} \sum_{i\in[N]} \frac{ \frac{K_i+\frac{\overline L_i K_i}{\rho_i}}{r} + \frac{\overline L_iK_i\beta_i}{\rho_i(1-\beta_i r)} }{ 1-\frac{\overline K_i-K_i}{r} }.
\]

Conversely, suppose that 
\[
\inf_{\max_j(\overline K_j-K_j)<r<\beta_{\max}^{-1}} \sum_{i\in[N]} \frac{ \frac{K_i+\frac{\overline L_i K_i}{\rho_i}}{r} + \frac{\overline L_iK_i\beta_i}{\rho_i(1-\beta_i r)} }{ 1-\frac{\overline K_i-K_i}{r} } <1. 
\]
Then there exists some $r\in \left( \max_j(\overline K_j-K_j), \beta_{\max}^{-1} \right)$ such that 
\begin{equation}\label{eq:contra}
\sum_{i\in[N]} \frac{ \frac{K_i+\frac{\overline L_i K_i}{\rho_i}}{r} + \frac{\overline L_iK_i\beta_i}{\rho_i(1-\beta_i r)} }{ 1-\frac{\overline K_i-K_i}{r} } <1.
\end{equation}
Suppose, toward a contradiction, that $\rho(B(r))\ge1$. By Lemma~\ref{lem:57}, 
\[
1= \sum_{i\in[N]} \frac{ \frac{K_i+\frac{\overline L_i K_i}{\rho_i}}{r} + \frac{\overline L_iK_i\beta_i}{\rho_i(1-\beta_i r)} }{ \rho(B(r))-\frac{\overline K_i-K_i}{r} }. \] Since $\rho(B(r))\ge1$, \[ \rho(B(r))-\frac{\overline K_i-K_i}{r} \ge 1-\frac{\overline K_i-K_i}{r}>0, 
\] and therefore \[ 1 \le \sum_{i\in[N]} \frac{ \frac{K_i+\frac{\overline L_i K_i}{\rho_i}}{r} + \frac{\overline L_iK_i\beta_i}{\rho_i(1-\beta_i r)} }{ 1-\frac{\overline K_i-K_i}{r} }, \] which contradicts \eqref{eq:contra}. Hence $\rho(B(r))<1.$ It follows that $ \inf_{0<r<\beta_{\max}^{-1}}\rho(B(r))<1,$ which proves the result. \end{proof}

\begin{proof}[Proof of Theorem \ref{thrm:4}]
    The result follows directly from Lemma \ref{lem:61} as we have $\lambda_{\infty}= \rho(B(1/z_0))$ by Lemma \ref{lem:55}, which implies
    \[
    \lambda_{\infty} = \inf_{ 0<r < \beta_{\max}^{-1}}\rho(B(r))
    \]
    by Lemma \ref{lem:8}, as desired.
\end{proof}

In Section \ref{sect:6}, we will use the value $t_*$ that realizes the minimization problem 
$$\inf_{\,0<r<(\max_i \beta_i)^{-1}}V(r)$$
to construct an exponential weight vector and use a majorization argument (via Collatz-Wielandt) to obtain a uniqueness result for discounted infinite-horizon non-stationary MFE via approximation through (discounted) finite-horizon MFE. For this uniqueness result, we will need that $\lambda_{\infty}<1$ and $t_*>1$. Numerically, it is easy to find $t_*$ that minimizes $V$ over the desired interval; hence, for a given system one can check both of these conditions numerically.
\begin{remark}
    In \cite{ayd} (also see \cite{aydin2025approximation}), a similar contraction condition is obtained for the single population case. The proof present in \cite[Section 3]{ayd} fails in the multi-population case. Although the techniques presented here can be applied to obtain an equivalent upper bound for the contraction rate obtained in \cite[Theorem 2]{ayd}, our techniques cannot describe a tight lower bound (that holds under any configuration) for $\rho(\mathcal S_T)$ when $T$ is sufficiently large as done in \cite[Corollary 2]{ayd}. Furthermore, the proof of \cite[Theorem 2]{ayd} indicates that in the single population group case the left eigenvectors of $\mathcal S_T$ have a constant ratio limit under normalization, which was proved only for the asymptotics of the ratios.
\end{remark}

\section{Finite-Horizon MFE to Infinite-Horizon Non-Stationary MFE: Convergence Rates and Uniqueness}\label{sect:6}
Let
\[
\bigl(X,\mathcal{P}(A),(C_i+\rho_i\Omega_i)_{i\in[N]},(P_i)_{i\in[N]},N\bigr)
\]
be a regularized multi-population MFG in which the objective function of each population \(i\in[N]\) is discounted by the factor \(\beta_i\). In this section, we will establish a uniqueness result for infinite-horizon non-stationary MFE and a finite-time error bound on the convergence of finite-horizon MFE to infinite-horizon non-stationary MFE. For this purpose, we will use the variational contraction function $V$ to construct weights and use a majorization argument under these weights on the matrices $(\mathcal S_T)_T$.

{Under Assumption \ref{ass:1}, the existence of a non-stationary MFE for a given initial-state measure $\tau_0$ follows directly from Schauder's fixed-point theorem.} For any $(\pmb \pi^{\infty},\pmb \tau^{\infty}) \in \mathrm{MFE}_{\tau_0}$, by $(\pmb \pi^{\infty} |_T,\pmb \tau^{\infty}|_T)$, we will denote the time truncation of this MFE to the first $T$-coordinates of every population group. The main result of this section is the following finite-time error bound, from which the uniqueness of the non-stationary MFE follows directly:

\begin{theorem}\label{thrm:conv}
    Let $(\pmb \pi^T,\pmb\tau^T) \in \mathrm{MFE}_{T,\tau_0}$. Suppose that there exists $t_o>1$ such that $V(t_o)<1$. Then, for all $t \ll T,$ we have
    \[
    \sum_{i \in [N]}\|\tau^{\infty}_{t,i}|_T-\tau^T_{t,i}\|_1 \le O\left( \frac 1{t^{T-t}_o}\right).
    \]
    In particular, $|\mathrm{MFE}_{\tau_0}|=1$, i.e., there exists a unique non-stationary MFE in the $N$-population case under any given family of initial state-measure $\tau_0$.
\end{theorem}

\begin{remark}
    If $V(1)<1$, then the working conditions of Theorem \ref{thrm:conv} are satisfied automatically due to the continuity of $V$. When $V(1)>1$ and $V(t_*)>1$, it is easy to see that the assumptions of Theorem \ref{thrm:conv} are not feasible. Thus, the interesting cases are when $V(1)>1$ and $V(t_*)<1$. 
\end{remark}

The proof of Theorem \ref{thrm:conv} relies on using the function $V$ to create a majorizing Lyapunov function as done in the proof of Lemma \ref{lem:8}. The construction of this Lyapunov function relies on the minimizers of the function $V$ used in Theorem \ref{thrm:quartic} to characterize $\lambda_\infty$. The main difference in this case will be that we will majorize the transpose of $\mathcal S_T$ to obtain these error bounds.

The condition $\lambda_\infty<1$ is required to control the differences $\tau^{\infty}_{t,i}|_T-\tau^T_{t,i}$ through any horizon length $T$, while we need $t_* >1$ to control the contribution of the remainder terms that appear due to the fact that $\pmb \tau^{\infty}$ induces an infinite-horizon horizon MFE.

Let $\bar t_*(T)=(t^{T-2}_*,t_*^{T-3},\cdots,t_*,1)$. By $c=(c_1,c_2,\cdots,c_N)$ denote a positive left Perron eigenvector of $B(t^*)$ and set $C=\sum_{i \in [N]}c_i$. Using these, we define 
\[
\hat t_*(T) = \underbrace{(c_1\bar t_*(T),c_2\bar t_*(T),\cdots,c_N\bar t_*(T))}_{N \text{ times}}.
\]For all $T \in \mathbb N$, we construct our Lyapunov function as
$
V(\tau,T) = \langle \hat t_* ,\tau\rangle.
$
The motivation behind the cyclic behavior of $\hat t_*$ can be readily seen from the proof of Lemma \ref{lem:8}, where we reduced the problem of finding a bound for $\rho(\mathcal S_T)$ to that of $\rho(B(r))$ using the special structure of eigenspace that corresponds to $\rho(B(r))$. To obtain Theorem \ref{thrm:conv}, we start with a version of Lemma \ref{lem:8} but for the transpose of $\mathcal S^{\top}_T$.

\begin{lemma}\label{lem:11}
    For all $n\in \mathbb N$, and $m \in [N],$ it holds that
\[
\max_{i}\frac{(\mathcal S^{\top}_n \hat t_*(n))_{(m-1)(n-1)+i}}{c_m\bar t_{*,i}(n)} \le \rho(B(t_*)).
\]
\end{lemma}
\begin{proof}
For $1\le i\le n-1$ we have $\bar t_{*,i+1}=\bar t_{*,i}/t_*$, and, for $1\le k\le i$,
$\bar t_{*,i-k+1}=t^{k-1}_*\bar t_{*,i}$. A direct inspection as done in Lemma \ref{lem:8} gives
\[
\frac{(\mathcal S^{\top}_n \hat t_{*}(n))_{(m-1)(n-1)+i}}{c_m\bar t_{*,i}(n)}
\le \rho(B(t_*)).
\]
The same also holds for $i=n-1$.
Thus, for all $i$, the desired inequality holds after combining the cases above.
\end{proof}

\begin{lemma}\label{lem:12}
    With the same notation of Lemma \ref{lem:11}, for any positive vector $y \in \mathbb R^{N(T-1)}$, we have that
    \[
    \langle \hat t_*(T),\mathcal S_T y \rangle \le \rho(B(t_*)) \langle \hat t_*(T),y\rangle.
    \]
\end{lemma}
\begin{proof}
    By Lemma \ref{lem:11}, we have
    \begin{equation}\label{eq:d3}
    \mathcal S^{\top}_T\hat t_*(T) \le \begin{bmatrix}
        \rho(B(t_*))c_1\bar t_*(T)\\
        \rho(B(t_*))c_2\bar t_*(T)\\
        \vdots\\
        \rho(B(t_*))c_N\bar t_*(T)
    \end{bmatrix} = \rho(B(t_*)) \hat t_*(T)
    , 
    \end{equation}
    where the last equality is due to Theorem \ref{thrm:quartic}. Since taking inner product with a positive vector $y$ will preserve \eqref{eq:d3}, the desired result follows.
\end{proof}

In the proof of Theorem \ref{thrm:conv}, we will denote  $\hat t_*(T)$ by $\hat t_*$ for simplicity.

\begin{proof}[Proof of Theorem \ref{thrm:conv}]
If $t_o \not = t_*$, then we can perturb $\mathcal S_T$ with a positive multiple of the identity matrix so that its spectral radius is $V(t_o)$. Thus, without loss of generality, we will suppose that $t_*>1$ for the remainder of this proof.

Let $\pmb \tau^T=(\tau^T_t)_{t=0}^{T-1}$ be a finite-horizon MFE obtained under the horizon length $T$ starting from $\tau_0$. Then, under the Euclidean inner product $\langle \cdot, \cdot \rangle$, arguing as in Lemma \ref{lem:4}, and using Lemma \ref{lem:12}, we obtain that
\begin{align*}
&\langle \hat t_*, \left(\|\tau^{\infty}_{t,j}|_T-\tau^T_{t,j}\|_1\right)_{t=1,j=1}^{T-1,N} \rangle 
\\&\le
\langle \hat t_*,\mathcal S_T \left(\| \tau^{\infty}_{t,j}|_T - \tau^T_{t,j}\|_1\right)_{t=1,j=1}^{T-1,N}\rangle +\langle \hat t_*, \bar{\mathcal S}\left( \underbrace{\|Q^{\pmb \tau^{\infty},i}_T\|_{\infty},\|Q^{\pmb \tau^{\infty},i}_T\|_{\infty},\cdots, \|Q^{\pmb \tau^{\infty},i}_T\|_{\infty}}_{T-1 \text{ times }}\right)_{i \in [N]}\rangle
\\&\le \rho(B(t_*)) \langle \hat t_*, \left(\| \tau^{\infty}_{t,j}|_T - \tau^T_{t,j}\|_1\right)_{t=1,j=1}^{T-1,N}\rangle 
\\& \qquad + \langle \hat t_*, \bar{\mathcal S}\left( \|Q^{\pmb \tau^{\infty},i}_T\|_{\infty},\|Q^{\pmb \tau^{\infty},i}_T\|_{\infty},\cdots, \|Q^{\pmb \tau^{\infty},i}_T\|_{\infty}\right)_{i \in [N]}\rangle,
\end{align*}
where $
    \bar{\mathcal S} 
    = 
    \begin{bmatrix} 
    \mathcal T_1 & \pmb 0 & \cdots & \pmb 0 \\
    \pmb 0 & \mathcal T_2 & \cdots & \pmb 0 \\
    \vdots & \vdots & \ddots & \vdots \\
    \pmb 0 & \pmb 0 & \cdots & \mathcal T_N
    \end{bmatrix},
$ $\mathcal T_i=\begin{bmatrix}
    \beta^T_i & 0 & \cdots & 0\\
    0 & \beta^{T-1}_i & \cdots & 0\\
    \vdots & \vdots & \ddots & \vdots \\
    0 & 0 & \cdots & \beta^2_i 
\end{bmatrix},$
and $\pmb 0$ is a $(T-1)\times (T-1)$ dimensional matrix made of $0$'s. Here, the terms $\|Q^{\pmb \tau,i}_{T}\|_{\infty}$ are due to the fact that $\pmb \tau^{\infty}$ induces an infinite-horizon MFE, which can be directly obtained from Lemma \ref{lem:1} through recursive iterations as we did in the proof of Lemma \ref{lem:4}.

Since $V(t_*)<1$, we can rearrange the inequality above to obtain that 
\begin{align*}
    &\min_i c_it^{T-t}_*\sum_{j \in [N]} \|\tau^{\infty}_{t,j}|_T-\tau^T_{t,j}\|_1\le\langle \hat t_* , \left(\|\tau^{\infty}_{t,j}|_T-\tau^T_{t,j}\|_1\right)_{t=1,j=1}^{T-1,N} \rangle 
    \\&\le \frac 1{1-\rho(B(t_*))}\underbrace{\langle \hat t_*, \bar{\mathcal S}\left(\|Q^{\pmb \tau^{\infty},i}_T\|_{\infty},\|Q^{\pmb \tau^{\infty},i}_T\|_{\infty},\cdots, \|Q^{\pmb \tau^{\infty},i}_T\|_{\infty}\right)_{i \in [N]}\rangle}_{=I}.
\end{align*}
It remains to control the term $I$. First, note that, as $t_*\beta _{\mathrm{argmax}_i \beta_i} < 1$, it holds that
\begin{align*}
I &\le N \max_i c_i\langle \bar t_*,\mathcal T_{\mathrm{argmax}_i \beta_i} \left( \underbrace{\max_i\|Q^{\pmb \tau^{\infty},i}_T\|_{\infty},\max_i\|Q^{\pmb \tau^{\infty},i}_T\|_{\infty},\cdots, \max_i\|Q^{\pmb \tau^{\infty},i}_T\|_{\infty}}_{T-1 \text{ times }}\right) \rangle
\\& \le N \max_i c_i \max_i\|Q^{\pmb \tau^{\infty},i}_T\|_{\infty}\beta_{\max} \sum_{j=1}^{T-1} \beta^{T-j}_{\max}t^{T-j}_* 
\\&\le N \max_i c_i\max_i\|Q^{\pmb \tau^{\infty},i}_T\|_{\infty}\beta^2_{\max} t_* \frac{\beta^{T-1}_{\max}t_*^{T-1}-1}{\beta_{\max} t_*-1}.
\end{align*}
Consequently, for the differences at section $t$ throughout all the population groups, and using the fact that $c_i>0$ for all $i \in [N]$, we obtain the following
\begin{align*}
    \sum_{j \in [N]} \|\tau^{\infty}_{t,j}|_T-\tau^T_{t,j}\|_1 \le  \frac N{1-\rho(B(t_*))} \frac{\max_i c_i}{\min_i c_i} \max_i\|Q^{\pmb \tau^{\infty},i}_T\|_{\infty}\beta^2_{\max} t^{t+1}_* \frac{\beta^{T-1}_{\max}t_*-\frac1{t_*^{T-t}}}{\beta_{\max} t_* -1}.
\end{align*}
Since $\beta_{\max}<1$ and $t_* >1$, as $T \to \infty$, we obtain that $\sum_{j \in [N]} \|\tau^{\infty}_{t,j}|_T-\tau^T_{t,j}\|_1 \to 0$. Furthermore, since $t_*\beta _{\max} < 1$, the dominating term in the sum above is $\frac{1}{t^{T-t}_*}$, which gives us the desired error bound.

It remains to prove the uniqueness of infinite-horizon non-stationary MFE under any given family of initial state-measures. For this purpose, we will demonstrate this for the initial state measure $\tau_0$, which does not break the generality. For each $T$, there exists a unique finite-horizon MFE as $\lim_{T \to \infty}\rho(\mathcal S_T)<1$. Furthermore, the error bound above shows that state-measures obtained under this MFE converge to any infinite-horizon MFE. From the uniqueness of the limit, it follows that all the infinite-horizon non-stationary MFE are obtained under a single state-measure flow $\pmb \tau^{\infty}$. Lastly, one can use dynamical programming to induce a unique family of $Q$-functions under $\pmb \tau^{\infty}$, see \cite[Proposition 3.10]{SaBaRaSIAM}. Due to the existence of the regularizer term in our setting, there is a unique policy induced by these $Q$-functions, which implies that there is a unique infinite-horizon MFE.
\end{proof}

In the single population case, we obtain the following formulation that is easier to handle:
\begin{corollary}
    Suppose $N=1$, i.e. we are in the single population case. If there exists $t_0>1$ such that
    \[
    \frac{\overline K_1+\frac{\overline L_1K_1}{\rho_1}}{t_0}+\frac{\overline L_1K_1\beta_1}{\rho_1}\frac{1}{1-\beta_1t_0} <1,
    \]
    then under any initial state-measure $\mu_0 \in \mathcal P(X)$ there exists a unique infinite-horizon non-stationary MFE.
\end{corollary}
\begin{proof}
    This follows directly from Theorem \ref{thrm:conv}.
\end{proof}

To obtain a rate of convergence between infinite-horizon non-stationary MFE and stationary MFE, we use the results in Section \ref{sect:stat}.

{
\begin{theorem}\label{thrm:7}
    Let $(\pmb \pi^{\infty},\pmb \tau^{\infty}) \in
    \mathrm{MFE}_{\tau_0}$ and let $(\pi,\tau)$ be a stationary MFE.
    If $V(1)<1$, then for all $k\geq0$,
    \[
    \sup_{i\in[N],\,t\geq k}
    \|\tau^\infty_{t,i}-\tau_i\|_1
    \leq O(\rho(B(1))^k).
    \]
\end{theorem}

\begin{proof}
First, we recall that, by Theorem~\ref{thrm:2},
    \[ \rho(\mathcal M)<1 \quad\Longleftrightarrow\quad \sum_{i\in[N]} \frac{K_i+\dfrac{\overline L_iK_i}{\rho_i(1-\beta_i)}}{1-(\overline K_i-K_i)}=V(1)<1. \]
    
    Let $Q^\infty_{t,i}$ be the optimal $Q$-function corresponding
    to the nonstationary flow $\pmb\tau^\infty$, and let
    $Q^{\tau,i}$ be the stationary optimal $Q$-function
    corresponding to $\tau$. By iterating \eqref{eq:q} along the
    Bellman equations, we obtain
    \[
    \|Q^\infty_{t,i}-Q^{\tau,i}\|_\infty \leq
    \overline L_i \sum_{\ell=0}^{\infty} \beta_i^\ell \sum_{j\in[N]} \|\tau^\infty_{t+\ell,j}-\tau_j\|_1.
    \]
    Therefore, for every $t\geq k$,
    \begin{align*}
    \|Q^\infty_{t,i}-Q^{\tau,i}\|_\infty
    &\leq
    \overline L_i \sum_{\ell=0}^{\infty} \beta_i^\ell \sum_{j\in[N]} \sup_{s\geq k} \|\tau^\infty_{s,j}-\tau_j\|_1
    \\&= \frac{\overline L_i}{1-\beta_i} \sum_{j\in[N]} \sup_{s\geq k} \|\tau^\infty_{s,j}-\tau_j\|_1.
    \end{align*}

    Since $\tau^\infty_{t+1,i} = H_{2,i}(Q^\infty_{t,i},\tau^\infty_t)$
    and $\tau_i = H_{2,i}(Q^{\tau,i},\tau),$
    Lemma~\ref{lem:1} gives, for every $t\geq k$,
    \begin{align*}
    \|\tau^\infty_{t+1,i}-\tau_i\|_1
    &\leq
    \frac{K_i}{\rho_i}
    \|Q^\infty_{t,i}-Q^{\tau,i}\|_\infty + \overline K_i \|\tau^\infty_{t,i}-\tau_i\|_1+ K_i
    \sum_{j\neq i}\|\tau^\infty_{t,j}-\tau_j\|_1
    \\
    &\leq \left(\overline K_i+ \frac{\overline L_iK_i}{\rho_i(1-\beta_i)} \right) \sup_{s\geq k}\|\tau^\infty_{s,i}-\tau_i\|_1
    +
    \left( K_i+ \frac{\overline L_iK_i}{\rho_i(1-\beta_i)} \right) \sum_{j\neq i} \sup_{s\geq k}
    \|\tau^\infty_{s,j}-\tau_j\|_1.
    \end{align*}
    Taking the supremum over $t\geq k$ yields
    \[
    \begin{bmatrix}
    \displaystyle
    \sup_{t\geq k+1}
    \|\tau^\infty_{t,1}-\tau_1\|_1
    \\
    \displaystyle
    \sup_{t\geq k+1}
    \|\tau^\infty_{t,2}-\tau_2\|_1
    \\
    \vdots
    \\
    \displaystyle
    \sup_{t\geq k+1}
    \|\tau^\infty_{t,N}-\tau_N\|_1
    \end{bmatrix}
    \leq
    \mathcal M
    \begin{bmatrix}
    \displaystyle
    \sup_{t\geq k}
    \|\tau^\infty_{t,1}-\tau_1\|_1
    \\
    \displaystyle
    \sup_{t\geq k}
    \|\tau^\infty_{t,2}-\tau_2\|_1
    \\
    \vdots
    \\
    \displaystyle
    \sup_{t\geq k}
    \|\tau^\infty_{t,N}-\tau_N\|_1
    \end{bmatrix}.
    \]

    Let $h=(h_1,\ldots,h_N)\in\mathbb R_{>0}^N$ be a positive
    right Perron eigenvector of $\mathcal M$, so that
    $\mathcal Mh=\rho(\mathcal M)h.$
    Define
    \[
    \|x\|_h
    :=
    \max_{i\in[N]}\frac{|x_i|}{h_i}.
    \]
    Since $\mathcal M$ is nonnegative, for every $x\geq0$,
    $\|\mathcal Mx\|_h\leq \rho(\mathcal M)\|x\|_h.$
    Consequently,
    \begin{align*}
    \left\|
    \begin{bmatrix}
    \displaystyle
    \sup_{t\geq k+1}
    \|\tau^\infty_{t,1}-\tau_1\|_1
    \\
    \vdots
    \\
    \displaystyle
    \sup_{t\geq k+1}
    \|\tau^\infty_{t,N}-\tau_N\|_1
    \end{bmatrix}
    \right\|_h
    \leq
    \rho(\mathcal M)
    \left\|
    \begin{bmatrix}
    \displaystyle
    \sup_{t\geq k}
    \|\tau^\infty_{t,1}-\tau_1\|_1
    \\
    \vdots
    \\
    \displaystyle
    \sup_{t\geq k}
    \|\tau^\infty_{t,N}-\tau_N\|_1
    \end{bmatrix}
    \right\|_h.
    \end{align*}
    Iterating this inequality gives
    \begin{align*}
    &
    \left\|
    \begin{bmatrix}
    \displaystyle
    \sup_{t\geq k}
    \|\tau^\infty_{t,1}-\tau_1\|_1
    \\
    \vdots
    \\
    \displaystyle
    \sup_{t\geq k}
    \|\tau^\infty_{t,N}-\tau_N\|_1
    \end{bmatrix}
    \right\|_h
    \leq
    \rho(\mathcal M)^k
    \left\|
    \begin{bmatrix}
    \displaystyle
    \sup_{t\geq0}
    \|\tau^\infty_{t,1}-\tau_1\|_1
    \\
    \vdots
    \\
    \displaystyle
    \sup_{t\geq0}
    \|\tau^\infty_{t,N}-\tau_N\|_1
    \end{bmatrix}
    \right\|_h.
    \end{align*}

    Let $h_{\min}:=\min_{i\in[N]}h_i,$ and $h_{\max}:=\max_{i\in[N]}h_i.$
    Since the $\ell_1$-distance between two probability measures is
    at most $2$, it follows that
    \[
    \left\|
    \begin{bmatrix}
    \displaystyle
    \sup_{t\geq0}
    \|\tau^\infty_{t,1}-\tau_1\|_1
    \\
    \vdots
    \\
    \displaystyle
    \sup_{t\geq0}
    \|\tau^\infty_{t,N}-\tau_N\|_1
    \end{bmatrix}
    \right\|_h
    \leq
    \frac{2}{h_{\min}}.
    \]
    Therefore,
    \[
    \sup_{i\in[N],\,t\geq k}
    \|\tau^\infty_{t,i}-\tau_i\|_1
    \leq
    2\frac{h_{\max}}{h_{\min}}\rho(\mathcal M)^k
    =
    O(\rho(\mathcal M)^k).
    \]
    Since \(\rho(\mathcal M) = \rho(B(1))\), we have thus obtained the desired inequality.
\end{proof}

A rate of convergence result between finite-horizon MFE and stationary MFE can be obtained directly from Theorems  \ref{thrm:conv} and \ref{thrm:7}.

\begin{corollary}
    Let $(\pi,\tau)$ be a stationary MFE and $(\pmb \pi^T,\pmb \tau^T) \in \mathrm{MFE}_{T,\tau_0}$. Then, we have that
    \[
    \| \tau^T_{k,i} - \tau_i\|_1 \le O\left( \rho(B(1))^k + \frac 1{t^{T-k}_*} \right)
    \]
\end{corollary}
\begin{proof}
    This follows directly from Theorems \ref{thrm:conv} and \ref{thrm:7} following a straightforward triangle inequality, and thus we do not include the details.
\end{proof}
}

\section{An Invariance Principle for Contraction under Slow-Fast Dynamic Updates for Lipschitz Systems}\label{sect:7}

The quasi-static algorithm in Section \ref{sect:stat} updates populations' $Q$-functions on a slower time scale than their state measures. As a result, the joint evolution of the $Q$-functions and state measures follows a slow-fast MFG dynamic. This naturally raises the question of whether the contraction condition established in the previous sections can be relaxed by allowing some populations to evolve on faster time scales than others.

For general dynamical systems, separating the dynamics into different time scales can improve convergence: quasi-static updates may converge to an equilibrium under conditions weaker than global contractivity of the full system; see \cite[Remark~2]{doan2025fast}. In this section, however, we show that this effect does not arise when the dynamics are governed by a nonnegative matrix. In that case, introducing slow-fast dynamics does not relax the contractivity requirement, regardless of the quasi-static regime considered. Since our contraction results are obtained by majorizing the original system with a dynamical system of this type, it follows that these contraction conditions cannot be improved by allowing some populations to update their parameters faster than others.

\begin{definition}
For each $i=1,2,\dots,N$, let $(X_i,d_{X_i})$ be a complete metric space. For each $j=1,2,\dots,N$, let
\[
f_j : \prod_{i=1}^N X_i \to X_j
\]
be a map satisfying
\[
d_{X_j}\bigl(f_j(x_1,\dots,x_N),\, f_j(\tilde x_1,\dots,\tilde x_N)\bigr)
\le
\sum_{i=1}^N a_{j,i}\, d_{X_i}(x_i,\tilde x_i),
\]
where $a_{j,i}\ge 0$ for all $i,j$.

We call the matrix $A=[a_{j,i}]_{j,i=1}^N$ the \emph{contraction matrix} of the family $(f_j)_{j=1}^N$ on $\prod_{i=1}^N (X_i,d_{X_i})$.
\end{definition}

Fix $0<N_1<N$. For any given
$
(x_j)_{j=N_1+1}^N \in \prod_{j=N_1+1}^N X_j,
$
we define a \emph{quasi-static update} with respect to the variables $(x_{N_1+j})_{j=1}^{N-N_1}$ to be any element
$
(x_1,\dots,x_{N_1})[x_{N_1+1},\dots,x_N] \in \prod_{i=1}^{N_1} X_i
$
satisfying
\begin{align*}
&(f_1,\dots,f_{N_1})\Bigl((x_1,\dots,x_{N_1})[x_{N_1+1},\dots,x_N],\,x_{N_1+1},\dots,x_N\Bigr) \\
&\qquad = (x_1,\dots,x_{N_1})[x_{N_1+1},\dots,x_N].
\end{align*}

Write
\[
A=
\begin{bmatrix}
A_1 & A_2\\
A_3 & A_4
\end{bmatrix},
\]
where $A_1\in\mathbb R^{N_1\times N_1}$ is the leading principal submatrix of $A$ of order $N_1$. For a given $(x_{N_1+j})_{j=1}^{N-N_1}$, the uniqueness of
$
(x_1,\dots,x_{N_1})[x_{N_1+1},\dots,x_N]
$
can be guaranteed by the condition $\rho(A_1)<1$.

Assume now that $\rho(A_1)<1$. Then the uniqueness of $(x_{N_1+1},\dots,x_N)$ satisfying
\[
(f_{N_1+1},\dots,f_N)\Bigl((x_1,\dots,x_{N_1})[x_{N_1+1},\dots,x_N],\,x_{N_1+1},\dots,x_N\Bigr)
=
(x_{N_1+1},\dots,x_N)
\]
can be guaranteed by the condition
$
\rho\bigl(A_4 + A_3(I-A_1)^{-1}A_2\bigr)<1.
$

Thus, updating the first $N_1$ coordinates of the iteration of $(f_i)_{i\in[N]}$ and then updating the remaining $N-N_1$ coordinates yields a slow-fast dynamical update scheme for the family $(f_i)_{i\in[N]}$. In the next result, we show that, in general,
\[
\rho(A)<1
\quad\Longleftrightarrow\quad
\rho(A_1)<1
\ \text{and}\
\rho\bigl(A_4 + A_3(I-A_1)^{-1}A_2\bigr)<1.
\]
In particular, using only Lipschitz coefficients, one cannot establish convergence of the quasi-static updates of $(f_i)_{i=1}^N$ to an equilibrium unless one can already prove convergence of the direct iteration of the maps $(f_i)_{i=1}^N$.

\begin{theorem}\label{thrm:8}
Let $m\in\mathbb N$, let $n_1,\dots,n_m\in\mathbb N$ be arbitrary, and set
$
N:=n_1+\cdots+n_m.
$
Let
$
A=[A_{ij}]_{i,j=1}^m \in \mathbb{R}^{N\times N}_{\ge 0},
A_{ij}\in \mathbb{R}^{\,n_i\times n_j}.
$
Thus $A$ is a nonnegative matrix partitioned into $m\times m$ blocks with arbitrary block
sizes, and each diagonal block $A_{ii}$ is automatically square.

Define recursively a sequence of nonnegative matrices $R_1,\dots,R_m$ as follows:
$
R_1:=A,
$
and for each $k=1,\dots,m-1$, write $R_k$ in the $2\times 2$ block form
\[
R_k=
\begin{bmatrix}
B_k & C_k\\
D_k & E_k
\end{bmatrix},
\]
where
$
B_k\in \mathbb{R}^{\,n_k\times n_k},
C_k\in \mathbb{R}^{\,n_k\times (n_{k+1}+\cdots+n_m)},
$
$
D_k\in \mathbb{R}^{\,(n_{k+1}+\cdots+n_m)\times n_k},$
$
E_k\in \mathbb{R}^{\,(n_{k+1}+\cdots+n_m)\times (n_{k+1}+\cdots+n_m)},
$
and define
$
R_{k+1}:=E_k+D_k(I-B_k)^{-1}C_k,
$
whenever $(I-B_k)^{-1}$ exists.
Then
$
\rho(A)<1
$ if, and only if, $\rho(B_k)<1$
for all $k=1,\dots,m-1,$ and $\rho(R_m)<1.$
\end{theorem}

\begin{proof}
We will proceed with induction on $m$.
First consider the case $m=2$. Then
\[
A=
\begin{bmatrix}
B_1 & C_1\\
D_1 & E_1
\end{bmatrix},
\]
where
$
B_1\in\mathbb R^{n_1\times n_1},
C_1\in\mathbb R^{n_1\times n_2},
D_1\in\mathbb R^{n_2\times n_1},
E_1\in\mathbb R^{n_2\times n_2},
$
and
$
R_2=E_1+D_1(I-B_1)^{-1}C_1.
$
We show that
$
\rho(A)<1
$ if and only if $
\rho(B_1)<1$ and $\rho(R_2)<1.$

Assume first that $\rho(B_1)<1$ and $\rho(R_2)<1$. Since $B_1\ge 0$ and $R_2\ge 0$,
the Neumann series imply
\[
(I-B_1)^{-1}=\sum_{\ell=0}^\infty B_1^\ell \ge 0,
\qquad
(I-R_2)^{-1}=\sum_{\ell=0}^\infty R_2^\ell \ge 0.
\]
Now
\[
I-A=
\begin{bmatrix}
I-B_1 & -C_1\\
-D_1 & I-E_1
\end{bmatrix},
\]
and the Schur complement of $I-B_1$ in $I-A$ is
\[
(I-E_1)-D_1(I-B_1)^{-1}C_1
=
I-R_2.
\]
Hence $I-A$ is invertible, and the block inverse formula yields
\[
(I-A)^{-1}
=
\begin{bmatrix}
(I-B_1)^{-1}+(I-B_1)^{-1}C_1(I-R_2)^{-1}D_1(I-B_1)^{-1}
&
(I-B_1)^{-1}C_1(I-R_2)^{-1}
\\[2mm]
(I-R_2)^{-1}D_1(I-B_1)^{-1}
&
(I-R_2)^{-1}
\end{bmatrix}.
\]
Thus, as $(I-B_1)^{-1}, (I-R_2)^{-1} \ge0$, it follows that $(I-A)^{-1}\ge0$.
Let $y\in\mathbb R^N$ be the vector of all ones and define
$x:=(I-A)^{-1}y.$ Since $(I-A)^{-1}$ exists, $I-A$ has no non-zero rows.
Then $x>0$ and
$(I-A)x=y>0,$
that is,
$
Ax=x-y<x.
$
Since $A\ge 0$, the existence of a vector $x>0$ such that $Ax<x$ implies
$\rho(A)<1$.

Conversely, assume that $\rho(A)<1$. Then
$(I-A)^{-1}=\sum_{\ell=0}^\infty A^\ell \ge 0.$
Let
$x:=(I-A)^{-1}y>0,$
and write $x=\binom{u}{v}$ conformably with the block decomposition of $A$, where
$u\in\mathbb R^{n_1},$ $v\in\mathbb R^{n_2}.$
Then
$Ax=x-y<x,$
so
$B_1u+C_1v<u,$
$D_1u+E_1v<v.$
In particular,
$B_1u<u.$
Since $B_1\ge 0$, this implies $\rho(B_1)<1$. Therefore $(I-B_1)^{-1}\ge 0$ exists, and from
$(I-B_1)u>C_1v$
we obtain
$u>(I-B_1)^{-1}C_1v.$
Multiplying by $D_1\ge 0$ gives
$D_1u>D_1(I-B_1)^{-1}C_1v.$
Using also $D_1u+E_1v<v$, we conclude that
$\bigl(E_1+D_1(I-B_1)^{-1}C_1\bigr)v<v,$
that is,
$R_2v<v.$
Since $R_2\ge 0$, this implies $\rho(R_2)<1$. Thus the claim holds for $m=2$.

Assume now that the theorem is true for all nonnegative block matrices with
$(m-1)\times(m-1)$ block structure, and let $A$ have $m\times m$ blocks. Write $A=R_1$ in the form
$R_1=
\begin{bmatrix}
B_1 & C_1\\
D_1 & E_1
\end{bmatrix},$
where
$B_1\in\mathbb R^{n_1\times n_1},$ 
$C_1\in\mathbb R^{n_1\times (n_2+\cdots+n_m)},$
$D_1\in\mathbb R^{(n_2+\cdots+n_m)\times n_1},$
$E_1\in\mathbb R^{(n_2+\cdots+n_m)\times (n_2+\cdots+n_m)}.$
By the already proved $m=2$ case,
$\rho(R_1)<1$ if and only if $\rho(B_1)<1$ and 
$\rho(R_2)<1,$
where
$R_2=E_1+D_1(I-B_1)^{-1}C_1.$
Now $R_2$ is again a nonnegative matrix with an $(m-1)\times(m-1)$ block structure whose
block sizes are $n_2,\dots,n_m$. Hence the induction hypothesis applies to $R_2$ and yields
$\rho(R_2)<1$ if and only if
$\rho(B_k)<1$ for $k=2,\dots,m-1,$
and
$\rho(R_m)<1.$
Combining the two equivalences gives
\[
\rho(A)=\rho(R_1)<1
\iff
\rho(B_1)<1,\ \rho(B_2)<1,\ \dots,\ \rho(B_{m-1})<1,\ \rho(R_m)<1.
\]
This completes the induction.
\end{proof}

From the perspective of stochastic approximation, Theorem \ref{thrm:8} shows that if the only available information about a system is its Lipschitz coefficients, then multi-time-scale stochastic approximation \cite{borkar1997stochastic} cannot be shown to converge unless the system is already known to converge under single-time-scale updates.

We now discuss several applications of Theorem \ref{thrm:8} to MFGs. In \cite{anahtarci2020value}, for single-population MFGs, consecutive updates of the $Q$-function and the state measure are used to compute a stationary MFE when the only available information about the system consists of Lipschitz bounds on its components. Theorems \ref{thrm:2} and \ref{thrm:8} show that the same conclusion remains valid in the multi-population setting: if one only knows the Lipschitz coefficients of the system, then using consecutive updates of the $Q$-functions and state measures to compute a stationary MFE converges under the same conditions as the quasi-static approach.
\begin{corollary}\label{cor:4}
Consider the map
\[
\left((Q_i)_{i=1}^N,\tau\right)
\in
\prod_{i=1}^N \mathcal C_i \times \prod_{i=1}^N \mathcal P(X)
\;\mapsto\;
\left((H_{1,i}(Q_i,\tau))_{i=1}^N,\,(H_{2,i}(Q_i,\tau))_{i=1}^N\right).
\]
Its contraction matrix on the product space
$
\prod_{i=1}^N (\mathcal C_i,\|\cdot\|_\infty)\times \prod_{i=1}^N (\mathcal P(X),\|\cdot\|_1)
$
has spectral radius strictly less than $1$ if and only if
\[
\sum_{i\in [N]}
\left(
K_i+\frac{\overline L_iK_i}{\rho_i(1-\beta_i)}
\right)
\frac{1}{1-(\overline K_i-K_i)}
<1.
\]
\end{corollary}
\begin{proof}
    This result directly follows from Theorems \ref{thrm:2} and \ref{thrm:8} since the $Q$-functions are updated in a quasi-static manner in Theorem \ref{thrm:2}.
\end{proof}

In the next corollary, we show that, in the regularized setting, if one uses only Lipschitz properties of the system, then allowing some populations to update their state measures faster than others does not yield a more relaxed contraction condition for the stationary MFG.

\begin{corollary}\label{cor:5}
Suppose that
$\{1,2,\dots,N\}
=
\underbrace{\{1,2,\dots,N_1\}}_{I_1}
\cup
\underbrace{\{N_1+1,\dots,N\}}_{I_2}.$
For any family of state measures $\pmb{\nu}\in \prod_{j=N_1+1}^N \mathcal P(X),$
define
$\pmb{\tau}(\pmb{\nu})\in \prod_{j=1}^{N_1}\mathcal P(X)$
to be the family of state measures satisfying
\[\pmb{\tau}(\pmb{\nu})
=
\bigl(H_{2,i}(Q^{(\pmb{\tau}(\pmb{\nu}),\pmb{\nu}),i},(\pmb{\tau}(\pmb{\nu}),\pmb{\nu}))\bigr)_{i=1}^{N_1}.
\]

Then the contraction matrix of the map
$$
\pmb{\tau} \mapsto \bigl(H_{2,i}(Q^{(\pmb{\tau},\pmb{\nu}),i},(\pmb{\tau},\pmb{\nu}))\bigr)_{i=1}^{N_1}
$$
on
$\prod_{j=1}^{N_1}(\mathcal P(X),\|\cdot\|_1),$
and the contraction matrix of the map
\[
\pmb{\nu} \mapsto \bigl(H_{2,i}(Q^{(\pmb{\tau}(\pmb{\nu}),\pmb{\nu}),i},(\pmb{\tau}(\pmb{\nu}),\pmb{\nu}))\bigr)_{i=N_1+1}^N
\]
on
$\prod_{j=N_1+1}^N(\mathcal P(X),\|\cdot\|_1),$
both have spectral radius less than $1$ if and only if
\[
\sum_{i\in [N]} \left(K_i+\frac{\overline L_i K_i}{\rho_i(1-\beta_i)}\right)\frac{1}{1-(\overline K_i-K_i)}<1.
\]
\end{corollary}
\begin{proof}
    This result directly follows from Theorems \ref{thrm:2} and \ref{thrm:8}.
\end{proof}

It follows directly from Lemma \ref{lem:4} that, in the multi-population finite-horizon setting, verifying contractivity of quasi-static updates with different time scales is already nontrivial even with only two populations. Indeed, one must check the spectral radius of
\[
T_2(\overline K_2) + T_2(K_2)\bigl(I - T_1(\overline K_1)\bigr)^{-1} T_1(K_1),
\]
which is more involved than directly controlling the matrix $\mathcal S_T$, due to the presence of the inverse term $\bigl(I - T_1(\overline K_1)\bigr)^{-1}$.

Another consequence of Theorem \ref{thrm:8} is that, even in the finite-horizon setting, the contraction criterion of Theorem \ref{thrm:4} cannot be sharpened using only the Lipschitz properties of the system. In particular, introducing heterogeneous update time scales across populations does not lead to stronger convergence guarantees.

\begin{corollary}\label{cor:6}
Suppose that
$\{1,2,\dots,N\}
=
\underbrace{\{1,2,\dots,N_1\}}_{I_1}\cup \underbrace{\{N_1+1,N_1+2,\dots,N\}}_{I_2}.$
For any
$\pmb \nu^T=(\nu_{t,i})_{t=0,\;i=N_1+1}^{T-1,\;N}\in\prod_{i=N_1+1}^N \{\nu_{0,i}\}\times\prod_{t=1}^{T-1}\prod_{i=N_1+1}^N \mathcal P(X),$
let
$\pmb \tau^T(\pmb \nu^T)\in\prod_{i=1}^{N_1}\{\tau_{0,i}\}\times\prod_{t=1}^{T-1}\prod_{i=1}^{N_1}\mathcal P(X)$
denote the family of state measures defined by
\[
\pmb \tau^T(\pmb \nu^T) = \left((\tau_{0,i})_{i=1}^{N_1},\left(H_{2,i}\bigl(Q_t^{(\pmb \tau(\pmb \nu),\pmb \nu),i},(\tau_t(\pmb \nu),\nu_t)\bigr)\right)_{i=1,\;t=1}^{N_1,\;T-1}\right).
\]

Then the contraction matrices of the maps
\[
\pmb \tau \mapsto \left((\tau_{0,i})_{i=1}^{N_1},\left(H_{2,i}\bigl(Q_t^{(\pmb \tau,\pmb \nu),i},(\tau_t,\nu_t)\bigr)\right)_{i=1,\;t=1}^{N_1,\;T-1}\right)
\]
and
\[
\pmb \nu \mapsto \left((\nu_{0,i})_{i=N_1+1}^{N},\left(H_{2,i}\bigl(Q_t^{(\pmb \tau(\pmb \nu),\pmb \nu),i},(\tau_t(\pmb \nu),\nu_t)\bigr)\right)_{i=N_1+1,\;t=1}^{N,\;T-1}\right)
\]
on
$\prod_{i=1}^{N_1}\{\tau_{0,i}\} \times \prod_{t=1}^{T-1}\prod_{i=1}^{N_1}(\mathcal P(X),\|\cdot\|_1)$
and
$
\prod_{i=N_1+1}^{N}\{\nu_{0,i}\} \times \prod_{t=1}^{T-1}\prod_{i=N_1+1}^{N}(\mathcal P(X),\|\cdot\|_1),$
respectively, have spectral radius less than \(1\) for all $T\ge 0$ if and only if
\[
\inf_{\overline K_\infty < r < \beta_{\max}^{-1}} \sum_{i\in[N]} \frac{ \frac{K_i+\frac{\overline{L_i}K_i}{\rho_i}}{r}+\frac{\overline L_i K_i\beta_i}{\rho_i(1-\beta_i r)}}{1-\frac{\overline K_i-K_i}{r}}<1.
\]
\end{corollary}
\begin{proof}
    This result directly follows from Theorems \ref{thrm:2} and \ref{thrm:4}.
\end{proof}

\section{Conclusion}
In this work, we have presented contraction conditions for multi-population discrete-time regularized MFGs. Using the variational nature of these contraction conditions, we have obtained error bounds between finite-horizon and infinite-horizon MFE. Furthermore, we have shown that these contraction conditions cannot be improved by introducing slow-fast dynamics when the only knowledge about the system is the Lipschitz coefficients.

The variational contraction conditions we have presented capture the asymptotics of the spectral radius that correspond to the contraction matrix in our setting, and provides a more general working conditions under which one can obtain error bounds between finite-horizon and infinite-horizon MFE compared to the one given in \cite{ayd}. However, our method does not provide a tight lower bound for the spectral radius in the finite-horizon setting unlike \cite[Theorem 2]{ayd}. We expect that spectral radius in the multi-population case should admit a lower bound similar to the one derived in \cite{ayd}. Establishing a tight lower bound for $\rho(\mathcal S_T)$ is an important direction for future research, as it would provide further insight into the asymptotic behavior of contraction rates for finite-horizon MFE in the multi-population setting.

A further research direction of interest would be to identify a continuous-state, continuous-action framework for studying the contraction properties of multi-population MFGs.
\nocite{aydin2025approximation}
\printbibliography
\appendix
\section{Proof of Lemma \ref{lem:1}}\label{app:a}
\begin{proof}[Proof of Lemma \ref{lem:1}]
First, we will prove the inequality \eqref{eq:q}. Using the triangle inequality, we obtain that
    \begin{align*}
&\|H_{1,i}(Q^1,\tau) - H_{1,i}(Q^2,\tilde \tau)\|_{\infty} 
\\&= \sup_{(x,u)\in X\times \mathcal P(A)}\bigg | C_i(x,u,\tau_1,\cdots,\tau_N) - C_i(x,u,\tilde \tau_1,\cdots,\tilde \tau_N) 
\\& \qquad \qquad + \beta_i \sum_{y \in X} Q^1_{\min}(y)P_i(y|x,u,\tau_1,\cdots,\tau_N)-\beta_i \sum_{y \in X} Q^2_{\min}(y)P_i(y|x,u,\tilde \tau_1,\cdots,\tilde \tau_k)\bigg|
\\& \leq \sup_{(x,u)\in X\times \mathcal P(A)} |C_i(x,u,\tau_1,\cdots,\tau_N)-C_i(x,u,\tilde \tau_1,\cdots,\tilde \tau_N)|
\\& \qquad + \sup_{(x,u) \in X\times \mathcal P(A)} \beta_i\underbrace{ \left | \sum_{y \in X} Q^1_{\min}(y)P_i(y|x,a,\tau_1,\cdots,\tau_N) - \sum_{y \in X} Q^2_{\min} P_i(y|x,a,\tilde \tau_1,\cdots,\tilde \tau_N) \right |}_{(\ast)}.
\end{align*}

We further bound the term $(\ast)$ as follows:
    \begin{align*}
    &\left | \sum_{y \in X} Q^1_{\min}(y)P_i(y|x,a,\tau_1,\cdots,\tau_N) - \sum_{y \in X} Q^2_{\min}P_i(y|x,a,\tilde \tau_1,\cdots,\tilde \tau_N) \right | 
    \\&\leq \left | \sum_{y \in X} Q^1_{\min}(y)\left(P_i(y|x,a,\tau_1,\cdots,\tau_N)- P_i(y|x,a,\tilde \tau_1,\cdots,\tilde \tau_N)\right)  \right|
    \\& \qquad + \left| \sum_{y \in X} (Q^1_{\min}-Q^2_{\min})(y) P_i(y|x,u,\tilde \tau_1,\cdots,\tilde \tau_N)\right|.
\end{align*}
Note that
\[
\left| \sum_{y \in X} (Q^1_{\min}-Q^2_{\min})(y) P_i(y|x,u,\tilde \tau_1,\cdots,\tilde \tau_N)\right| \le \| Q^1_{\min}- Q^2_{\min}\|_{\infty}.
\]
Let $\| Q^1_{\min}\|_{\mathrm{Lip}}$ be the smallest Lipschitz coefficient of $Q^1_{\min}$ over $X$. Using the concentration inequality \cite[Lemma A.2]{kontorovich2008concentration}, we directly obtain
\begin{align*}
& \left | \sum_{y \in X} Q^1_{\min}(y)\left(P_i(y|x,a,\tau_1,\cdots,\tau_N)- P_i(y|x,a,\tilde \tau_1,\cdots,\tilde \tau_N)\right)  \right| 
\\& \le \frac{\| Q^1_{\min}\|_{\mathrm{Lip}}}{2} \| P_i(\cdot|x,a,\tau_1,\cdots,\tau_N) - P_i(\cdot|x,a,\tilde \tau_1,\cdots,\tilde \tau_N) \|_1
\\&\le \frac{K_i}2\| Q^1_{\min} \|_{\mathrm{Lip}}\sum_{j \in [N]}\|\tau_j-\tilde \tau_j\|_{1},
\end{align*}
where the last line is due to Assumption \ref{ass:1}. Thus, one obtains
\begin{align*}
    &\left | \sum_{y \in X} Q^1_{\min}(y)\left( P_i(y|x,a,\tau_1,\cdots,\tau_N)- P_i(y|x,a,\tilde\tau_1,\cdots,\tilde \tau_N)\right)  \right| 
    \\& \qquad + \left| \sum_{y \in X} (Q^1_{\min}-Q^2_{\min})(y) P_i(y|x,u,\tilde \tau_1,\cdots,\tilde \tau_N)\right|
    \\& \le \frac{K_i}2\| Q^1_{\min} \|_{\mathrm{Lip}}\sum_{j \in [N]}\|\tau_j-\tilde \tau_j\|_{1} + \| Q^1_{\min}- Q^2_{\min}\|_{\infty},
\end{align*}
It then follows that
\begin{align*}
&\sup_{(x,u) \in X \times \mathcal P(A)}\beta_i \left | \sum_{y \in X} Q^1_{\min}(y)P_i(y|x,a,\tau_1,\cdots,\tau_N) - \sum_{y \in X} Q^2_{\min}(y)P_i(y|x,a,\tilde \tau_1,\cdots,\tilde \tau_N) \right |
\\& \leq \sup_{(x,u) \in X \times \mathcal P(A)}\beta_i K_i\| Q^1_{\min} \|_{\mathrm{Lip}}\sum_{j \in [N]}\|\tau_j-\tilde \tau_j\|_{1} + \beta_i\| Q^1_{\min}- Q^2_{\min}\|_{\infty}
\\& \le \frac{K_i\beta_i}2\| Q^1_{\min} \|_{\mathrm{Lip}}\sum_{j \in [N]}\| \tau_j - \tilde \tau_j\|_1 + \beta_i \| Q^1_{\min}-Q^2_{\min}\|_{\infty}.
\end{align*}
Combining all the terms above, we then obtain
\[
\|H_{1,i}(Q^1,\tau) - H_{1,i}(Q^2,\tilde \tau)\|_{\infty} \le\left(L_i+\frac{K_i\beta_i}2\|Q^1_{\min}\|_{\mathrm{Lip}}\right)\sum_{j \in [N]}\|\tau_j-\tilde \tau_j\|_1+\beta_i \| Q^1_{\min}- Q^2_{\min}\|_{\infty}.
\]
It remains to calculate $\|Q^1_{\min}\|_{\mathrm{Lip}}$. First, note that $\|Q^1_{\min}\|_{\mathrm{Lip}}\le \sup_{u \in \mathcal P(A)}\sup_{x \not = y} |Q^1(x,u)-Q^1(y,u)|.$ Since $Q^1(x,u) = \sum_{a \in A}\mathfrak q^1(x,a)u(a)+\rho_i\Omega_i$ for some $\mathfrak q^1 \in \tilde{\mathcal C}_i$, directly from the definition of $\mathfrak q^1$ we obtain
\[
\sup_{u \in \mathcal P(A)} |Q^1(x,u)-Q^1(y,u)|\le \sup_{a \in A} \| \mathfrak q^1(x,\cdot)\|_{\mathrm{Lip}} \le \frac{L_i}{1-\frac{\beta_i K_i}{2}},
\]
from which \eqref{eq:q} follows.

For the rest of the proof, we will focus on showing that \eqref{eq:m} holds. We start with the triangle inequality:
    \begin{align*}
        &\|H_{2,i}(Q^1,\tau)- H_{2,i}(Q^2,\tilde \tau)\|_1 
        \\&= \sum_y \left|\sum_x P_i(y|x, \textrm{argmin}_{u \in \mathcal P(A)} Q^1(x,u),\tau_1,\cdots,\tau_N)\tau_i(x) - \sum_x P_i(y|x,\textrm{argmin}_{u \in \mathcal P(A)} Q^2(x,u), \tilde \tau_1,\cdots,\tilde \tau_N) \tilde \tau_i(x)\right|
        \\& \le \sum_y \left|\sum_x P(y|x,\textrm{argmin}_{u \in \mathcal P(A)} Q^1(x,u),\tau_1,\cdots,\tau_N)\tau_i(x) - \sum_x P_i(y|x, \textrm{argmin}_{u \in \mathcal P(A)} Q^2(x,u), \tilde \tau_1,\cdots,\tilde \tau_n)\tau_i(x)\right| 
        \\& \qquad +\sum_y \left|\sum_x P_i(y|x, \textrm{argmin}_{u \in \mathcal P(A)} Q^2(x,u), \tilde \tau_1,\dots,\tilde \tau_N)\tau_i(x) - \sum_x P_i(y|x, \textrm{argmin}_{u \in \mathcal P(A)} Q^2(x,u), \tilde \tau_1,\cdots,\tilde \tau_N) \tilde \tau_i(x)\right|.
    \end{align*}
    Consequently, using Assumption \ref{ass:1}, we must have
    \begin{align*}
    &\|H_{2,i}(Q^1,\tau_1,\cdots,\tau_N)- H_{2,i}(Q^2,\tilde \tau_1,\cdots,\tilde \tau_N)\|_1 
    \\&\leq \sum_x \left\| P_i(\cdot|x,\textrm{argmin}_{u \in \mathcal P(A)} Q^1(x,u),\tau_1,\cdots,\tau_N) - P_i(\cdot|x,\textrm{argmin}_{u \in \mathcal P(A)} Q^2(x,u), \tilde \tau_1,\cdots,\tilde \tau_N) \right\|_1 \tau_i(x) 
    \\&\qquad + \frac{K_i}2\left( 1 + \frac{\overline L_i}{\rho_i}\right) \|\tau_i-\tilde{\tau}_i\|_1 \\
    &\leq K_1\left( \sup_x \left\|\textrm{argmin}_{u \in \mathcal P(A)} Q^1(x,u) -\textrm{argmin}_{u \in \mathcal P(A)} Q^2(x,u) \right\|_1 + \sum_{j \in [N]}\|\tau_j - \tilde \tau_j\|_1 \right)
    \\& \qquad+\frac{K_i}2\left( 1 + \frac{\overline L_i}{\rho_i}\right) \|\tau_i-\tilde \tau_i\|_1 \\
    &\leq \frac{K_i}{\rho_i}\| Q^1 - Q^2\|_{\infty}+\overline K_i \|\tau_i-\tilde \tau_i\|_{1} + K_i \sum_{j \not = i} \| \tau_j-\tilde \tau_j\|_1,
\end{align*}
where on the last line we have used that 
\[
 \sup_x \left\|\textrm{argmin}_{u \in \mathcal P(A)} Q^1(x,u) -\textrm{argmin}_{u \in \mathcal P(A)} Q^2(x,u) \right\|_1 \le \frac{1}{\rho_i} \| Q^1-Q^2\|_{\infty};
\]
see \cite[Proposition 3]{anahtarci2023q}. In particular, we have obtained \eqref{eq:m}.
\end{proof}

\end{document}